\documentclass[letterpaper, 11pt, reqno]{amsart}

\usepackage{amsmath,amssymb,amscd,amsthm,amsxtra, esint, color}
\usepackage{mathrsfs} 

\usepackage[dvips]{graphicx}

\allowdisplaybreaks[2]

\newtheorem{theorem}{Theorem} [section]

\newtheorem{lemma}[theorem]{Lemma}
\newtheorem{proposition}[theorem]{Proposition}
\newtheorem*{proposition*}{Proposition}
\newtheorem{remark}[theorem]{Remark}

\newtheorem{definition}[theorem]{Definition}
\newtheorem{corollary}[theorem]{Corollary}

\newcommand{\noi}{\noindent}
\newcommand{\Z}{\mathbb{Z}}
\newcommand{\R}{\mathbb{R}}
\newcommand{\C}{\mathbb{C}}
\newcommand{\T}{\mathbb{T}}

\let\Re=\undefined\DeclareMathOperator*{\Re}{Re}
\let\Im=\undefined\DeclareMathOperator*{\Im}{Im}

\newcommand{\al}{\alpha}

\newcommand{\dl}{\delta}

\newcommand{\eps}{\varepsilon}

\newcommand{\ld}{\lambda}

\newcommand{\ft}{\widehat}

\newcommand{\cj}{\overline}

\renewcommand{\o}{\omega}

\newcommand{\les}{\lesssim}
\newcommand{\ges}{\gtrsim}

\newcommand{\jb}[1]
{\langle #1 \rangle}

\renewcommand{\b}{\beta}

\newcommand{\pa}{\partial}

\newcommand{\M}{\mathcal{M}}
\def\e{\varepsilon}

\newtheorem*{ackno}{Acknowledgement}

\numberwithin{equation}{section}
\numberwithin{theorem}{section}

\newsavebox{\foobox}

\begin{document}
\baselineskip = 14pt


\title[Ill-posedness of cubic fractional NLS on the real line]
{Ill-posedness of the cubic nonlinear half-wave equation and other fractional NLS on the real line}

\author[A.~Choffrut and O.~Pocovnicu]
{Antoine Choffrut and Oana Pocovnicu}

\address{
Antoine Choffrut\\
School of Mathematics\\
The University of Edinburgh\\
James Clerk Maxwell Building\\
The King's Buildings\\
Peter Guthrie Tait Road\\
Edinburgh\\
EH9 3FD\\
Scotland and
Mathematics Institute\\
University of Warwick\\
Zeeman Building\\
Coventry CV4 7AL\\
United Kingdom}
 
\email{A.Choffrut@warwick.ac.uk}

\address{
Oana Pocovnicu\\
Department of Mathematics\\ 
Princeton University\\ 
Fine Hall, Washington Rd.\\
Princeton, NJ 08544-1000, USA, and
Department of Mathematics\\
Heriot-Watt University
and The Maxwell Institute for the Mathematical Sciences\\
Edinburgh EH14 4AS\\
United Kingdom}

\email{o.pocovnicu@hw.ac.uk}



\begin{abstract}
In this paper, we study ill-posedness of cubic fractional nonlinear Schr\"odinger equations. 
First, we consider the cubic nonlinear half-wave equation (NHW) on $\R$. 
In particular, we prove the following ill-posedness results:
(i) failure of local uniform continuity
of the solution map in $H^s(\R)$ for $s\in (0,\frac 12)$, 
and also for $s=0$ in the focusing case;
(ii) failure of $C^3$-smoothness of the solution map in $L^2(\R)$;
(iii) norm inflation and, in particular, failure of continuity of the solution map
in $H^s(\R)$, $s<0$.
By a similar argument, we also prove norm inflation in negative Sobolev spaces
for the cubic fractional NLS. 
Surprisingly, we obtain norm inflation 
{\it above} the scaling critical regularity in the case of dispersion $|D|^\beta$ with $\beta>2$. 
\end{abstract}

\maketitle

%

\date{\today}

%
%

\baselineskip = 15pt

\section{Introduction}\label{SEC:intro}

In this paper, we consider the fractional nonlinear Schr\"odinger equations (NLS) on $\R$:
\begin{equation}\label{FNLS0}
\begin{cases}
i\pa_t u-|D|^\beta u=\mu |u|^2u\\
u(0)=u_0, 
\end{cases} \quad (t,x)\in \R\times \R,
\end{equation}
where $u$ is complex-valued, $\beta>0$,
$\ft{|D|^\beta f}(\xi)=|\xi|^\beta \ft f(\xi)$ for all $\xi\in\R$, and $\mu \in\{-1,1\}$. 
When $\mu=1$ the equation \eqref{HW} is called defocusing, while 
it corresponds to the focusing equation when $\mu=-1$.
PDEs with a nonlocal dispersion, such as \eqref{FNLS0},
appear in various physical contexts. 
For example, such one-dimensional PDEs were proposed by  
Majda, McLaughlin, and Tabak
as models for wave turbulence in \cite{MMT}. 
Other physical instances 
include water waves \cite{IP}, continuum limits of lattice points \cite{KLS}, and gravitational collapse \cite{ES, FL}.

A major part of this paper is dedicated to the study of ill-posedness of the
following cubic nonlinear half-wave equation (NHW) on $\R$, 
corresponding to the case $\beta=1$
in \eqref{FNLS0}:
\begin{equation}\label{HW}
\begin{cases}
i\pa_t u-|D| u=\mu |u|^2u\\
u(0)=u_0.
\end{cases}
\end{equation}
In particular, we prove different forms of ill-posedness 
depending on the phase space $H^s(\R)$;
see Theorem~\ref{THM:main}.
In the last section, 
we return to the cubic fractional NLS \eqref{FNLS0} and prove a strong form of ill-posedness
in negative Sobolev spaces; see Theorem~\ref{fracNLS}.

We start by recalling briefly the notions of local and global well-posedness for \eqref{FNLS0}.
\begin{definition}\label{DEF:WP}
{\rm 
We say that the Cauchy problem \eqref{FNLS0} is locally well-posed in $H^s(\R)$
if the following two conditions hold:

\noi
(i) For every bounded set $B$ of $H^s(\R)$,
there exist $T>0$ and a Banach space $X_T$ continuously embedded in $C([-T,T];H^s(\R))$
such that if $u_0\in B$, 
then there exists a unique solution $u$ of \eqref{FNLS0} 
on $[-T,T]$ in $X_T$. 

\noi
(ii) The solution map
$\Phi_\mu: u_0\mapsto u$ is continuous from $B$ to $C([-T,T];H^s(\R))$.

If $T$ can be chosen arbitrarily large, 
then we say that the Cauchy problem \eqref{FNLS0} is globally well-posed in $H^s(\R)$. 
}
\end{definition}

\subsection{The cubic nonlinear half-wave equation}
NHW
arises as a Hamiltonian evolution associated to the energy: 
\begin{equation*}
E(u(t)):= \int \frac 12||D|^{\frac 12}u(t,x)|^2+\frac{\mu}{4} |u(t,x)|^4 dx.
\end{equation*}
In particular, the energy is conserved (at least formally)
by the evolution: $E(u(t))=E(u(0))$ for all $t\in\R$.
Notice that the energy $E(u)$ is positive-definite when $\mu=1$ and, moreover, 
it controls the $\dot{H}^{\frac 12}(\R)$-norm of $u$,
while 
it is not sign definite when $\mu=-1$. 
Not surprisingly, as we see in Proposition \ref{PROP:WP} below, 
$\mu$ plays a crucial
role in the global-in-time dynamics of \eqref{HW}.
Another important conservation law of \eqref{HW}
is the mass:
\begin{equation*}
M(u(t)):=\int_\R |u(t,x)|^2 dx=M(u(0)).
\end{equation*}
The symmetries associated to the two conservation laws above,
as guaranteed by Noether's theorem,
are the time translation $u(t,x)\mapsto u(t+t_0,x)$
 and the phase rotation $u(t,x)\mapsto e^{i\theta}u(t,x)$. 
Another crucial symmetry of \eqref{HW}
is the scaling invariance. Namely, if $u$ is a solution of \eqref{HW},
then 
\[u_\ld(t,x)=\ld^{\frac 12}u(\ld t, \ld x)\]
is also a solution of \eqref{HW} with rescaled initial data $u_\ld(0,x)=\ld^{\frac 12}u(0,\ld x)$.
Notice that 
the scaling leaves the $L^2$-norm of a solution invariant,
$\|u_\ld (t, \cdot)\|_{L^2}=\|u(\ld t, \cdot)\|_{L^2}$. 
Namely, \eqref{HW} is said to be
``$L^2$-critical" or ``mass-critical".

We first recall the
well-posedness theory for \eqref{HW}
from \cite{GG12, KLR, Poc2}.
\begin{proposition}[\cite{GG12, KLR, Poc2}, Well-posedness theory for NHW] 
\label{PROP:WP}

{\rm
Let $s\geq \frac 12$.

The Cauchy problem \eqref{HW} is locally well-posed in $H^s(\R)$.
Moreover, for $s>\frac 12$, the solution map $\Phi_\mu$
is Lipschitz continuous on bounded subsets of $H^s(\R)$. 

In the defocusing case ($\mu=1$), \eqref{HW} is globally well-posed in $H^s(\R)$.
In the focusing case ($\mu=-1$), \eqref{HW} is globally well-posed for initial data in $H^s(\R)$
of sufficiently small mass, while there exist finite time blowup solutions of large mass.  
}
\end{proposition}

\subsection{Ill-posedness of the cubic nonlinear half-wave equation}

One of our goals in the present paper
is to complete the local-in-time study of the Cauchy problem
for NHW on $\R$.
In particular, we complement Proposition  \ref{PROP:WP}
by proving various ill-posedness results
in $H^s(\R)$ for $s< \frac 12$.

Given $s\in\R$, in the following we denote by 
$\Phi_\mu: u_0\mapsto u$ the a priori defined solution map of \eqref{HW} 
from $H^s\cap \mathcal S(\R)$ to $C([-T,T];H^s(\R)\cap \mathcal S(\R))$ (see Definition \ref{DEF:WP} above). 
(Here, $\mathcal S(\R)$ denotes the space of Schwartz functions on $\R$.)
Given $t\in\R$, we also denote by $\Phi_\mu(t):u_0\mapsto u(t)$ the solution map acting on 
$H^s(\R)\cap \mathcal S(\R)$.

\begin{theorem}[Ill-posedness of NHW on $\R$]\label{THM:main} \textcolor{white}a\\
\noi
{\rm (i)} {\rm(}Failure of local uniform continuity, $0<s<\frac 12${\rm)}. 
Let $0<s<\frac 12$.
$\Phi_\mu$ fails to be uniformly continuous on bounded sets of $H^s(\R)$, for $\mu\in\{-1,1\}$. 
More precisely, given $0<\eps\ll 1$, there exist global solutions of \eqref{HW}
$u_1^\eps$ and $u_2^\eps$ such that $\|u_1^\e(0)\|_{H^s}\les 1$, $\|u_2^\e(0)\|_{H^s}\les 1$,
\begin{equation*}
\lim_{\eps\to 0}\| u_1^\eps(0)- u_2^\eps(0)\|_{H^s}=0,
\end{equation*}
and
\begin{equation*}
\liminf_{\eps\to 0}\| u_1^\eps- u_2^\eps\|_{L^\infty([0,T];H^s)}\gtrsim 1 \quad \text { for all } \quad T>0.
\end{equation*}

\medskip
\noi
{\rm (ii)} {\rm(}Focusing {\rm NHW}, $s=0${\rm)}. $\Phi_{-1}$ fails to be uniformly continuous on bounded sets of $L^2(\R)$. More precisely, a statement similar to that in {\rm (i)} holds with $\mu=-1$ and $s=0$. 

\medskip

\noi
{\rm (iii)} {\rm(}Failure of $C^3$-smoothness, $s=0${\rm)}. 
Fix $0< t\leq 1$.
$\Phi_\mu(t)$ fails to be $C^3$-smooth on $L^2(\R)$ for $\mu\in\{-1,1\}$.

\medskip

\noi
{\rm (iv)} {\rm(}Norm inflation in $H^s$, $s<0${\rm)}. Let $s<0$. 
Given $0<\eps \ll 1$, there exist a solution $u^\eps \in C(\R; H^{\infty}(\R))$ and 
$0<t_\eps<\eps$ such that 
\[ \|u^\eps(0)\|_{H^s(\R)}<\eps \qquad \text{and} \qquad \|u^\eps(t_\eps)\|_{H^s(\R)}>\frac{1}{\eps}.\] 

In particular, $\Phi_\mu$ fails to be continuous at zero in $H^s(\R)$, $s<0$, for $\mu\in\{-1,1\}$. 
\end{theorem}

In the periodic setting,
Georgiev, Tzvetkov, and Visciglia \cite{GTV} 
recently proved the failure of local uniform continuity of the solution map
of NHW
in $H^s(\T)$, $s\in (\frac 14, \frac 12)$. 
Our proof of Theorem~\ref{THM:main} (i)
is inspired by their work.
There are, however, some 
differences discussed in Subsection \ref{subsec:Rem} below,
allowing us to cover a larger range of regularities $s\in (0,\frac 12)$.

In view of Definition \ref{DEF:WP},
we note that parts (i)-(iii) of Theorem~\ref{THM:main}  refer to a `mild' form of ill-posedness, 
namely the failure of local uniform continuity or $C^3$-smoothness of the solution map. 
In particular, it might still be possible to construct a locally continuous flow in $H^s(\R)$
for $s\in [0,\frac 12)$. 
Parts (i)-(iii), on the other hand, show that
such a local well-posedness result in $H^s(\R)$, $s\in [0,\frac 12)$, 
cannot be proved using a
fixed point argument.

Theorem~\ref{THM:main} (iii) states that the solution map of the defocusing NHW
fails to be $C^3$-smooth on $L^2(\R)$.
We do not know, however, whether the solution map 
fails to be uniformly continuous on bounded sets of $L^2(\R)$. 
This remains an interesting open question.
See Remark \ref{defocNHW_L2} for a detailed discussion.

As remarked in the statement of Theorem~\ref{THM:main},
the norm inflation in part (iv) is a stronger property than
the failure of continuity of the solution map $\Phi_\mu$ at zero in $H^s(\R)$. 
Indeed, for the failure of continuity of $\Phi_\mu$ at zero in $H^s(\R)$
it suffices to find a solution $u^\eps \in C(\R, H^{\infty}(\R))$ and $0<t_\eps<\eps$ such that 
$\|u^\eps (0)\|_{H^s}<\eps$
and $\|u^\eps (t_\eps)\|_{H^s}\ges 1$. 
(Here, $H^\infty(\R)=\cap_{s>0}H^s(\R)$.)

\medskip

\subsection{Relation to the cubic Szeg\H o equation}
An important feature of the nonlinear half-wave equation \eqref{HW}
is the fact that the resonant equation associated to it is  
a completely integrable model, namely the cubic Szeg\H{o} equation:
\begin{align}\label{SZ}
\begin{cases}
i\pa_t V=\Pi_+(|V|^2V)\\
V(0)=V_0,
\end{cases}
\end{align}
where $\ft{\Pi_+f}(\xi)=\ft{f}(\xi)\pmb{1}_{\xi\geq 0}$
is the Szeg\H{o} projector. 
In particular, 
the dynamics of NHW can be well approximated, 
for a long time, by that of the Szeg\H{o} equation.
See \cite{Poc2} and Proposition \ref{PROP:approx} below.
This approximation plays an important role in the proof of Theorem \ref{THM:main} (i). 

We recall that the cubic Szeg\H{o} equation was introduced by G\'erard and Grellier in 
\cite{GG10} on $\T$,
and was studied on $\R$ by the second author in \cite{Poc0, Poc1}.
The Szeg\H{o} equation on $\R$ is globally well-posed in
\[H^s_+(\R):=L^2_+(\R)\cap H^s(\R), \qquad s\geq \tfrac 12,\]
where 
$L^2_+(\R)= \big \{f\in L^2(\R): \text{supp}\, \ft f\subset [0,\infty)\big\}$.
We remark that, via a Paley-Wiener theorem,
$L^2_+(\R)$ can also be identified with the 
Hardy space of holomorphic functions in the upper-half plane 
$\C_+:=\{z: \Im z>0\}$. 

As a byproduct of the proof of Theorem \ref{THM:main},
we also have the following ill-posedness result 
for the cubic Szeg\H{o} equation \eqref{SZ}. 
\begin{proposition}[Ill-posedness of the cubic Szeg\H{o} equation on $\R$]\label{final rem:SZ}
\textcolor{white}a\\
{\rm
\noi
{\rm (i)} Let $0\leq s<\frac 12$. The solution map of the cubic Szeg\H{o} equation
fails to be uniformly continuous on bounded sets in $H^s_+(\R)$.

\smallskip
\noi
{\rm (ii)} Let $s<0$. The cubic Szeg\H{o} equation \eqref{SZ}
has the norm inflation property in $H^s_+(\R)$
{\rm(}in 
the sense of part (iv) of Theorem \ref{THM:main}{\rm)}. 
In particular, the solution map of the cubic Szeg\H{o} equation \eqref{SZ} 
fails to be continuous at zero in $H^s_+(\R)$, $s<0$.
}
\end{proposition}

Along with the above mentioned approximation of NHW by the Szeg\"o equation,
Proposition \ref{final rem:SZ} (i) (for $0<s<\frac 12$)
is the key ingredient in proving Theorem \ref{THM:main} (i).
On the other hand, Proposition \ref{final rem:SZ} (ii) follows 
by slightly modifying the argument used to prove Theorem \ref{THM:main} (iv).

\subsection{Norm inflation in negative Sobolev spaces for fractional NLS}

Now we turn our attention to the cubic fractional NLS \eqref{FNLS0}
with a general value of $\beta>0$.
The equation \eqref{FNLS0}
possesses the scaling symmetry 
$u\mapsto u_\ld(t,x):=\ld^{\frac{\beta}{2}}u(\ld x, \ld^\beta t)$.
Associated to this symmetry, one 
defines the scaling critical regularity $s_{\rm crit}$
to be the index $s$
for which the $\dot H^s$-norm
of $u(0)$ is invariant under this scaling.
A simple calculation yields that 
$$s_{\rm crit}=\frac{1-\b}{2}\,.$$

The well-posedness theory of the cubic fractional NLS \eqref{FNLS0} with $1<\beta<2$ was studied 
in \cite{CHKL, DET}.
In particular, it was shown that \eqref{FNLS0} 
is locally well-posed
in $H^s$ for $s\geq \frac{2-\beta}{4}$
both on $\R$ and $\T$.
Moreover, \cite{CHKL}
proved the failure of local uniform continuity of the solution map 
of \eqref{FNLS0} on $\R$ in $H^s(\R)$, $\frac{2-3\beta}{4(\beta+1)}<s<\frac{2-\beta}{4}$,
in the case $1<\beta<2$. 

Proceeding as in Theorem \ref{THM:main} (iv),
we prove the following norm inflation property (a strong form of ill-posedness) for 
fractional NLS with a general $\beta>0$. 

\begin{theorem}[Norm inflation property for fractional NLS]\label{fracNLS}

The cubic fractional {\rm NLS} \eqref{FNLS0}
has the norm inflation property in $H^s(\R)$ {\rm(}in the sense of
Theorem \ref{THM:main} {\rm (iv)}{\rm)} 
in the following cases {\rm(}see the shaded region in Figure \ref{Fig1}{\rm)}:
\begin{itemize}
\item $0<\beta< 1$: $s< 0$,
\item $1\leq \b <2$: $s<s_{\rm crit}$,
\item $\b= 2$: $s\leq s_{\rm crit}$,
\item  $\b>2$: $s<\frac{1-2\b}{6}$.
\end{itemize}
\end{theorem}

\begin{figure}[h]  

\includegraphics{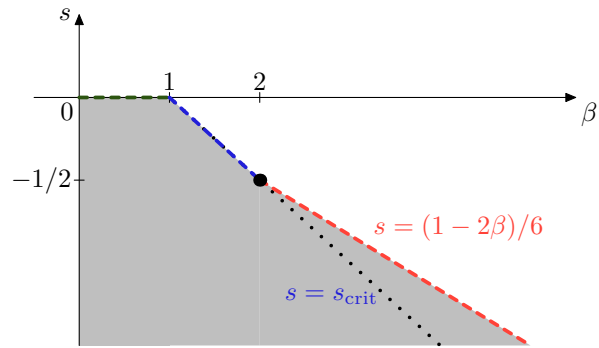}

\caption{Region of norm inflation for fractional NLS.}

\label{Fig1}

\end{figure}

It is generally conjectured that a PDE is ill-posed in $H^s$
for $s<s_{\rm crit}$. 
For cubic NLS ($\beta=2$) on $\R$ and $\T$,
Christ, Colliander, and Tao \cite{CCTarxiv} and Kishimoto \cite{Kishimoto1}
proved, indeed,
norm inflation in $H^s(\R)$, $s\leq s_{\rm crit}=-\frac 12$.
In Theorem \ref{fracNLS}, we prove norm inflation 
in $H^s(\R)$, $s<s_{\rm crit}$, for all fractional NLS with $1\leq \beta\leq 2$.
Surprisingly, in the case $\beta>2$, 
we obtain norm inflation for regularities $s<\frac{1-2\beta}{6}$,
where we note that $\frac{1-2\beta}{6}>s_{\rm crit}$.
To the best of the authors' knowledge, 
this is the first result of norm inflation
{\it above} the scaling critical regularity for NLS-type equations
with a gauge-invariant nonlinearity. 

In the non gauge-invariant case, previously 
Iwabuchi and Ogawa \cite{Iwabuchi_Ogawa} and 
Iwabuchi and Uriya \cite{Iwabuchi_Uriya}
obtained norm inflation above the scaling critical regularity for 
NLS with nonlinearities $u^2$ on $\R$, and $|u|^2$ on $\R^n$, $n=1,2,3$, respectively. 
Our norm inflation results
are inspired by \cite{Iwabuchi_Ogawa} and 
by a further development by Kishimoto \cite{Kishimoto1}.

In the gauge-invariant case, there are results on other types of 
ill-posedness above the scaling critical regularity.
We explain below why Theorem \ref{fracNLS} with $\beta>2$ is of a different nature
compared to these previous results.
Previously, 
Christ, Colliander, and Tao \cite{CCTarxiv2} 
and Molinet \cite{Molinet} showed the failure of continuity of the solution map in $H^s(\T)$, $s<0$,
for the cubic NLS on $\T$ ($\b=2$ and $s_{\rm crit}=-\frac 12$).
Furthermore, Guo and Oh \cite{Guo_Oh} 
proved non-existence of solutions of the cubic NLS on $\T$
if the initial data is in $H^s(\T)\setminus L^2(\T)$, where $s\in (-\frac 18, 0)$.
The latter also holds for the cubic biharmonic NLS ($\beta=4$) on $\T$
(with $s_{\rm crit}=-\frac 32$).
We remark that all these ill-posedness results no longer apply
if we remove a certain resonant term from the nonlinearity. 
More precisely, set $\M(u):=\frac{1}{2\pi}\int_\T |u|^2dx$ and consider the Wick ordered cubic fractional NLS on $\T$:
\[i\partial_t u-|D|^\beta u =\left(|u|^2-\M(u)\right)u.\]
Then the above-mentioned ill-posedness results in \cite{CCTarxiv2, Molinet, Guo_Oh}
do not apply to the Wick ordered 
cubic NLS, nor to the Wick ordered cubic biharmonic NLS.

Let us now turn to Theorem \ref{fracNLS}.
First, we make the observation that the proof of Theorem \ref{fracNLS}
for \eqref{FNLS0} on $\R$ can be carried over to $\T$ almost literally. 
One can then use the gauge transformation 
to map a solution $u$ of the cubic fractional NLS \eqref{FNLS0} on $\T$ 
with initial data $u_0\in L^2(\T)$,
into a solution of the Wick ordered cubic fractional NLS on $\T$:
\[u\mapsto \mathcal G(u)(t) := e^{-2it \M (u_0)} u(t).\]
Note that $\mathcal G$ preserves any $H^s$-norm.
Since we only work with smooth initial data,
this immediately allows us to deduce that
the norm inflation in Theorem \ref{fracNLS}
also holds for the Wick ordered cubic fractional NLS on $\T$.
In particular, this shows that simply removing $\M(u_0) u$ from the nonlinearity
of the cubic fractional NLS on $\T$ is not sufficient to make the solution map continuous. 
Instead, in our case, the discontinuity
of the solution map 
is more critical and is due to a genuinely nonlinear effect 
(the ``high-to-low frequency cascade" in the nonlinearity; see Subsection \ref{subsec:Rem} below). 
To the best of the authors' knowledge, this is a new phenomenon 
for regularities larger than the scaling critical regularity, in the case of gauge-invariant nonlinearities.

\subsection{Comments on the proofs and further remarks}\label{subsec:Rem}
\textcolor{white}a

In a recent work,
Georgiev, Tzvetkov, and Visciglia \cite{GTV} 
proved the failure of local uniform continuity of the solution map
of NHW \eqref{HW} on the torus $\T$
in $H^s(\T)$, $s\in (\frac 14, \frac 12)$. 
In proving Theorem \ref{THM:main} (i),
we follow the strategy of \cite{GTV}.
Namely, given $0<\eps\ll1$, 
we consider two solutions $u_1^\eps, u_2^\eps$ of NHW
that can be respectively approximated by explicit traveling waves $v_1^\eps, v_2^\eps$
of the Szeg\H{o} equation \eqref{SZ}.
The two traveling waves $v_1^\eps, v_2^\eps$
are chosen in such a way that they are very close to each other
at time $t=0$, but after a long time their distance increases 
{\it relatively} to their initial distance (remaining still small). 
Then, the solutions $u_1^\eps, u_2^\eps$ of NHW 
inherit the same properties regarding their distance. 
Finally, exploiting the scaling symmetry of NHW on $\R$,
the distance between the solutions $u_1^\eps, u_2^\eps$  of NHW
can be made $\ges 1$ at a
very small time.
While the result in \cite{GTV} in the periodic case holds 
in $H^s(\T)$, $s\in (\frac 14, \frac 12)$, 
we were able to prove failure of the local uniform continuity of the solution map
of NHW on $\R$
in $H^s(\R)$ for $s\in (0, \frac 12)$. 
We enlarged the range of regularities
by using a finer approximation of NHW by the Szeg\H{o} equation,
together with the scaling symmetry of NHW on $\R$ (not available on $\T$).
See Remark \ref{REM:GTV} below for details.

For the proof of Theorem \ref{THM:main} (ii),
we work directly with two traveling waves of the focusing NHW on $\R$,
rather than working with solutions that can be approximated 
by explicit solutions of the Szeg\H{o} equation. 
Note that this argument is not applicable to the defocusing NHW, 
since there are no traveling waves in that case. 
We choose the two speeds of the traveling waves of focusing NHW 
in such a way that: 
(i) at time $t=0$
the traveling waves are very close to each other, 
and
(ii) at a later time
they are spatially located far away from each other
and, therefore, their distance becomes $\ges 1$.
This type of argument 
was first used 
in the work of Birnir, Kenig, Ponce, Svanstedt, and Vega \cite{BKPSV}. 
It was later used by Kenig, Ponce, and Vega \cite{KPV01}
to prove the failure of local uniform continuity
of the solution map for the focusing cubic NLS on $\R$.
See also \cite{Biagioni_Linares, Herr_Lenzmann}.

The failure of $C^3$-smoothness of the solution map in Theorem \ref{THM:main} (iii)
follows from the unboundedness of the trilinear operator $\nabla^3 \Phi_\mu\big|_{u_0=0}$. 
This type of approach was first introduced by Bourgain in \cite{Bourgain97}, 
where he proves the failure of $C^3$-smoothness
of the solution map for KdV and mKdV on both $\T$ and $\R$
below certain threshold regularity. 
See also \cite{Tzvetkov_CRM}. 

The proof of Theorem \ref{THM:main} (iv)
is based on analyzing each term appearing in the Picard iteration scheme for NHW \eqref{HW}.
In particular, we show that the 
cubic term $U_3$
(see equation \eqref{Uk})
dominates all the other terms and is unbounded in $H^s(\R)$.  
To achieve this goal we exploit the ``high-to-low energy cascade"
in the nonlinearity. 
This idea appeared first in the work of Bejenaru and Tao \cite{BT}
as an abstract and general argument for proving ill-posedness. 
Recently,  
Iwabuchi and Ogawa \cite{Iwabuchi_Ogawa}
developed this idea further,
making it more easily applicable to a wider class of equations. 
In this paper,
we follow closely an argument of Kishimoto \cite{Kishimoto1}. 
We also exploit the abundance of resonances of NHW. 
In Theorem \ref{fracNLS},
we prove an analogous result for the cubic fractional NLS with general $\beta>0$. 
More care is needed in this case since 
\eqref{FNLS0} with $\beta\neq 1$
has less resonances than NHW.

\begin{remark}
{\rm 
In \cite{Ohgeneric}, Oh further extended the approach of 
Iwabuchi and Ogawa \cite{Iwabuchi_Ogawa}
to prove a norm inflation phenomenon 
based at general initial conditions
(not only the zero initial condition) for the
cubic NLS on $\R^d$ and $\T^d$. 
It was remarked in \cite{Ohgeneric}
that a similar argument
can be used to extend our results from Theorem \ref{THM:main} (iv)
and Theorem \ref{fracNLS} to
norm inflation 
based at general initial conditions.
See also \cite{Xia} for a previous 
result on generic ill-posedness. 

}
\end{remark}

\begin{remark}
{\rm
Theorem \ref{THM:main} (iv) and Theorem \ref{fracNLS}
also hold on $\T$ with essentially the same proof. 
In a recent work, Oh and Wang \cite{Oh} obtained a similar result
for the fractional NLS on $\T$.
On the one hand, their result does not include 
regularities above $s_{\rm crit}$.
On the other hand, it includes
positive regularities $0<s<s_{\rm crit}$ (for $0<\b<1$), 
that are not considered in this paper.
The strategy of \cite{Oh} is different from ours.
It consists in 
exploiting the 
``high-to-low energy cascade" in the dynamics of dispersionless NLS
$i\pa_t u=|u|^2u$,
and then 
in approximating dispersionless NLS by the 
small dispersion fractional NLS 
$i\pa_t u-\nu |D|^\b u=|u|^2u$, $0<\nu\ll1$. 
This strategy was first used by Christ, Colliander, and Tao
\cite{CCT_AJM, CCTarxiv} in the context of NLS on $\R^d$.
See also \cite{BGT, BTI}.
In \cite{Oh}, this was adapted to the periodic setting
by using a modified scaling argument to relate small dispersion fractional NLS to fractional NLS. 
See also \cite{Kishimoto1, Carles_Kappeler} for recent results 
on norm inflation for NLS in the periodic setting. 
}
\end{remark}

Finally, we refer to \cite{BGT1, Koch_Tzvetkov, Tzvetkov_GAKUTO, Lindblad, 
Lebeau, Lebeau1, Ibrahim, Alazard_Carles, CDS1, CDS2} for more ill-posedness results
for nonlinear PDEs.

{\it Organization of the paper.} 
Sections 2-5 are dedicated to the proof of Theorem~\ref{THM:main}. 
More precisely, the four parts of Theorem~\ref{THM:main} follow from
Propositions~\ref{PROP:UC}, \ref{PROP:UCL2}, \ref{PROP:C3}, and \ref{PROP:inflation},
respectively.
We also note that Proposition \ref{final rem:SZ} on ill-posedness of the cubic Szeg\H{o} equation
follows from Corollary \ref{COR:UC.SZ}
and Remarks \ref{REM:C3 SZ} and \ref{REM:inflation SZ}.
Finally, in Section 6 we prove Theorem \ref{fracNLS}.


\section{Failure of local uniform continuity of the solution map of NHW in $H^s(\R)$, $s\in (0, \frac 12)$}

In this section, we prove Theorem \ref{THM:main} (i), namely the failure of local uniform continuity of the solution map of \eqref{HW} with $\mu=\pm 1$, with respect to the topology of $H^s(\R)$, $s\in (0, \frac 12)$. 

As it was recalled in the introduction, 
solutions of equation \eqref{HW} with small, well-prepared data 
can be approximated for a long time by solutions of the completely integrable 
cubic Szeg\H{o} equation \eqref{SZ}.
See Proposition \ref{PROP:approx} below.
Thus, we will use explicit traveling waves of this integrable 
model to first show the failure of local uniform continuity 
 for the Szeg\H{o} equation. 
 Next, we will use the above mentioned approximation
 to deduce the failure of local uniform continuity for \eqref{HW}.

We start with an elementary lemma that we use repeatedly in this paper. 
We denote by $\mathcal F$ the Fourier transform:
\[\mathcal  F f(\xi)=\hat{f} (\xi):=\int_{\R}e^{-ix\xi}f(x)dx \qquad \text{ for all } \xi\in\R.\]
For $s\in\R$, we define the homogeneous $\dot H^s$-norm by
\[\|f\|_{\dot H^s(\R)}:=\frac{1}{\sqrt{2\pi}}\left(\int_\R |\xi|^{2s}|\ft f(\xi)|^2d\xi\right)^{\frac 12},\]
while the non homogeneous $H^s$-norm is defined by
\[\|f\|_{H^s(\R)}:=\frac{1}{\sqrt{2\pi}}\left(\int_\R \jb{\xi}^{2s}|\ft f(\xi)|^2d\xi\right)^{\frac 12},\]
where $\jb{\xi}=\sqrt{1+\xi^2}$.
\begin{lemma}\label{Fourier}
Let $c\in\R$ and $p>0$. 
Then the following hold:

\noi
{\rm (i).} $\mathcal{F}\left(\frac{1}{x-ct+ip}\right) (\xi) = -2\pi i e^{-ict\xi} e^{-p\xi} \,   \pmb{1}_{\xi>0}$
for all $\xi\in\R$.

\noi
{\rm (ii).} For any $s>-\frac 12$, we have $\left\|\frac{1}{x-ct+ip}\right\|_{\dot{H}^s}=\frac{\sqrt{2\pi \Gamma(2s+1)}}{(2p)^{s+\frac 12}}$, where $\Gamma$ denotes the gamma function.

\noi
{\rm (iii).} $\frac{1}{x-ct+ip}\notin H^s(\R)$ for any $s\leq -\frac 12$.
\end{lemma}

\begin{proof}
Part (i) follows directly from the residue theorem, while part (ii) follows
from (i)
and by the change of variables $\eta=2p\xi$:
\begin{align*}
\left\|\frac{1}{x-ct+ip}\right\|_{\dot{H}^s}^2&=\frac{1}{2\pi}\int |\xi|^{2s}\left|\mathcal{F}\left(\frac{1}{x-ct+ip}\right) (\xi)\right|^2d\xi
=2\pi\int_0^\infty \xi^{2s}e^{-2p\xi}d\xi\\
&=\frac{2\pi}{(2p)^{2s+1}}\int_0^\infty \eta^{2s}e^{-\eta}d\eta
=\frac{2\pi \Gamma (2s+1)}{(2p)^{2s+1}}. 
\end{align*}
For part (iii), it suffices to notice that $\int_0^\infty \eta^{2s}e^{-\eta}d\eta=\infty$ for any $s\leq -\frac 12$.
\end{proof}

Next we recall the following classification of traveling waves of the Szeg\H{o} equation on $\R$ from \cite{Poc1}. These are special solutions of the Szeg\H{o} equation of the form $V(t,x)=e^{-i\omega t}V_0(x-ct)$.

\begin{proposition}[\cite{Poc1}, Traveling waves of the Szeg\H{o} equation on $\R$]
A function $u\in C(\R; H^{\frac 12}_+)$ is a traveling wave solution of the Szeg\H{o} equation on $\R$ if and only if there exist
$\phi, a\in \R$ and $\alpha, p>0$ such that
\begin{equation*}
V(t,x)=\frac{\alpha e^{i\phi}e^{-i\omega t}}{x-ct+a+ip} \quad \quad  \text{ for all } t\in\R,
\end{equation*}
where
\begin{equation}\label{c_omega}
c:=\frac{\alpha^2}{2p} \quad \text{ and } \quad \omega:=\frac{\alpha^2}{4p^2}
\end{equation}
\end{proposition}

Next, we introduce a basic construction of two traveling waves of the Szeg\H{o} equation
whose distance between each other exhibits 
a relative growth in time. 
This basic construction plays an essential role in the 
proof of Theorem \ref{THM:main} (i). 
Moreover, this immediately yields the failure of local uniform continuity 
of the solution map for the Szeg\H{o} equation in $H^s_+(\R)$, $s\in [0,\frac 12)$, 
as shown in Corollary \ref{COR:UC.SZ} below.

\begin{proposition}[Basic construction for the Szeg\H{o} equation]\label{PROP:SZ}
Let $s> -\frac 12$ and $\delta>0$. 
Given $0<\eps\ll 1$, there exist global solutions 
$\tilde V_1^\eps, \tilde V_2^\eps\in C(\R;H^s_+(\R))$ of \eqref{SZ} such that
\begin{equation}\label{SZ1}
\|\tilde V_1^\e(0)\|_{H^s_+}+\|\tilde V_2^\e(0)\|_{H^s_+}\les \eps, \qquad \|\tilde V_1^\eps(0)-\tilde V_2^\eps(0)\|_{H^s_+}\sim \eps|\log\eps|^{-\frac 12},
\end{equation}
and
\begin{equation}\label{SZ2}
\|\tilde V_1^\eps(t)-\tilde V_2^\eps(t)\|_{H^s_+}\ges \eps
\end{equation}
for all $t\geq \frac{\delta}{\eps^2}|\log\eps|$.
\end{proposition} 

The role of the parameter $\dl$ will be clear in the proof of
Proposition \ref{PROP:UC}, where it is chosen sufficiently small so as to apply
the approximation result of Proposition \ref{PROP:approx}.
\begin{proof}[Proof of Proposition \ref{PROP:SZ}]
We choose $\tilde V_j^\eps$ to be the following traveling waves of the Szeg\H{o} equation on $\R$:
\begin{equation}\label{V_j}
\tilde V_j^\eps(t,x):=\frac{\alpha_j e^{-i\omega_j t}}{x-c_jt+ip},
\end{equation}
where 
\begin{equation}\label{p_al}
p:=1, \qquad \alpha_1:=\eps, \qquad \alpha_2:=\eps(1+|\log\eps|^{-\frac 12}).
\end{equation}
By \eqref{c_omega}, notice that we have
\begin{align}\label{c_2-c_1}
c_1=\frac{\eps^{2}}{2}, \quad c_2=\frac{\eps^{2}(1+2|\log\eps|^{-\frac 12}+ |\log\eps|^{-1})}{2},
\quad c_2-c_1=\eps^{2}|\log\eps|^{-\frac 12}(1+o(1)).
\end{align}

\noi
Then, by Lemma \ref{Fourier}, it follows that
\begin{align*}
\|\tilde V_1^\eps(0)-\tilde V_2^\eps(0)\|_{H^s_+}=(\alpha_2-\alpha_1)\left\|\frac{1}{x+ip}\right\|_{H^s_+}
\sim \eps |\log\eps|^{-\frac 12},
\end{align*}
$\|\tilde V_1^\e(0)\|_{H^s_+}+\|\tilde V_2^\e(0)\|_{H^s_+}\les \eps$, and thus \eqref{SZ1} is satisfied.
Then, using again Lemma \ref{Fourier}, we have
\begin{align}
\|\tilde V_1^{\eps}(t)-\tilde V_2^\eps(t)\|_{\dot H^s_+}^2
&=\|\tilde V_1^\eps(t)\|_{\dot H^s_+}^2+\|\tilde V_2^\eps\|_{\dot H^s_+}^2
-4\pi\al_1\al_2\Re\left(e^{i(\o_2-\o_1)t}\int_0^\infty \xi^{2s}e^{-2p\xi}e^{i\xi(c_2-c_1)t}d\xi\right)\notag\\
&=\eps^2\frac{\pi \Gamma(2s+1)}{2^{2s-1}}(1+o(1))-A,\label{Anew}
\end{align}
where
\begin{equation*}
A:=4\pi\al_1\al_2\Re\left(e^{i(\o_2-\o_1)t}\int_0^\infty \xi^{2s}e^{-2p\xi}e^{i\xi(c_2-c_1)t}d\xi\right).
\end{equation*}
Next, we show that $|A|\ll \eps^2$.
This comes down to finding an expression for 
$\int_0^\infty x^ae^{-\ld x}dx$ for $a>-1$ and $\Re \ld>0$. If $\ld\in\R_+$,
then by a change of variables we have
\begin{align*}
\int_0^\infty x^ae^{-\ld x}dx=\frac{1}{\ld^{a+1}}\int_0^\infty x^ae^{-x}dx=\frac{\Gamma (a+1)}{\ld^{a+1}}.
\end{align*}
Since both $\int_0^\infty  x^ae^{-\ld x}dx$ and $\frac{\Gamma (a+1)}{\ld^{a+1}}$ are holomorphic in $\ld$
for $\Re \ld>0$ and since they coincide on $\R_+$, 
it follows that they coincide on $\{\ld: \Re\ld>0\}$.
Therefore,
\begin{align*}
\int_0^\infty \xi^{2s}e^{-\xi(2p-i(c_2-c_1)t)}d\xi=\frac{\Gamma(2s+1)}{(2p-i(c_2-c_1)t)^{2s+1}}.
\end{align*}
Taking $t\geq \frac{\delta}{\eps^2}|\log\eps|$ and using \eqref{c_2-c_1}, we get that
\begin{align*}
|A|\lesssim \frac{\eps^2}{|2-i(c_2-c_1)t|^{2s+1}}\lesssim \frac{\eps^2}{(\delta|\log\eps|^{\frac 12})^{2s+1}}\ll \eps^2.
\end{align*}
Combining this with \eqref{Anew}, we thus obtain that 
\begin{align*}
\|\tilde V_1^{\eps}(t)-\tilde V_2^\eps(t)\|_{\dot H^s_+}\sim_s \eps,
\end{align*}
which yields \eqref{SZ2}.
\end{proof}

\begin{corollary}
[Failure of local uniform continuity for SZ in $H^s_+(\R)$, $s\in [0,\frac 12)$]
\label{COR:UC.SZ}
Let $s\in [0,\frac 12)$. 
Given $0<\eps\ll 1$, there exist global solutions
$ V_1^\eps, V_2^\eps\in C(\R;H^s_+(\R))$ of \eqref{SZ} such that
$\| V_1^\e(0)\|_{H^s_+}+\| V_2^\e(0)\|_{H^s_+}\les 1$,
\begin{equation*}
\lim_{\eps\to 0}\| V_1^\eps(0)- V_2^\eps(0)\|_{H^s_+}=0,
\end{equation*}
and
\begin{equation*}
\liminf_{\eps\to 0}\| V_1^\eps- V_2^\eps\|_{L^\infty([0,T];H^s_+)}\gtrsim 1 \quad \text { for all } \quad T>0.
\end{equation*}
\end{corollary}

\begin{proof}
First, we consider the case $s\in (0,\frac 12)$. We define
\[ V_j^\eps(t,x):=(\tilde V_j^\eps)_\ld (t,x)=\ld^{\frac 12}\tilde V_j^\eps(\ld t,\ld x),\]
where $\tilde V_j^\eps$ are as in Proposition \ref{PROP:SZ} with $\dl=1$, $j=1,2$, and $\ld=\eps^{-\frac 1s}$. 
The functions $V_j^\eps$ are still solutions of the Szeg\H{o} equation.
Notice that $\| V_1^\eps(t)- V_2^\eps(t)\|_{L^2_+}=\|\tilde V_1^\eps(\ld t)-\tilde V_2^\eps(\ld t)\|_{L^2_+}$,
while $\| V_1^\eps(t)- V_2^\eps(t)\|_{\dot H^s_+}=\eps^{-1}\|\tilde V_1^\eps(\ld t)-\tilde V_2^\eps(\ld t)\|_{\dot H^s_+}$ for all $t\in\R$. 
Combining these with \eqref{SZ1} and \eqref{SZ2}, we obtain
\begin{align*}
\| V_1^\eps(0)\|_{H^s_+}+\|V_2^\eps(0)\|_{H^s_+}&\les_s 1\\
 \| V_1^\eps(0)- V_2^\eps(0)\|_{H^s_+}&\sim_s |\log\eps|^{-\frac 12}\ll1,\\
\| V_1^\eps(t)- V_2^\eps(t)\|_{\dot H^s_+}&\sim_s 1
\end{align*}
for $t\geq \eps^{\frac 1s-2}|\log\eps|$.
Noting that $\eps^{\frac 1s-2}|\log\eps|\to 0$ as $\eps\to 0$ precisely when $0<s<\frac 12$
concludes the proof in the case $s\in (0,\frac 12)$.

Next, we turn to the case $s=0$. 
Because of the singularity at $s=0$ in the
scaling $V_j^\eps=(\tilde V_j^\eps)_\ld$, $\ld=\eps^{-\frac 1s}$,
we can no longer use the above approach. 
Instead, here we consider
\begin{align*}
V_j^\eps(t,x):=\frac{\al_je^{-i\o_jt}}{x-c_jt+ip},
\end{align*}
where
\begin{align*}
p:=\eps, \qquad \al_1:=\eps^{\frac 12}, \qquad \al_2:=\eps^{\frac 12}(1+|\log\eps|^{-\frac 12}).
\end{align*}
By \eqref{c_omega}, we notice that
\begin{align*}
c_1=\frac 12, \qquad c_2=\frac{(1+|\log\eps|^{-\frac 12})^2}{2}, \qquad c_2-c_1=|\log\eps|^{-\frac 12}(1+o(1)).
\end{align*} 
By Lemma \ref{Fourier}, we have that $\|V_1^\eps(0)\|_{L^2_+}=\sqrt{\pi}$,
$\|V_2^\eps(0)\|_{L^2_+}=\sqrt{\pi}(1+|\log\eps|^{-\frac 12})$, and
\begin{align*}
\|V_2^\eps(0)-V_1^\eps(0)\|_{L^2_+}=\eps^{\frac 12}|\log\eps|^{-\frac 12}\left\|\frac{1}{x+i\eps}\right\|_{L^2_+}\sim |\log\eps|^{-\frac 12}\ll 1.
\end{align*}
Moreover,
\begin{align*}
\| V_1^{\eps}(t)- V_2^\eps(t)\|_{L^2_+}^2
&=\|V_1^\eps(t)\|_{L^2_+}^2+\| V_2^\eps\|_{L^2_+}^2
-A'\geq 2\pi -A',
\end{align*}
where
\begin{equation*}
A':=4\pi\al_1\al_2\Re\left(e^{i(\o_2-\o_1)t}\int_0^\infty e^{-2p\xi}e^{i\xi(c_2-c_1)t}d\xi\right).
\end{equation*}
Arguing as in the proof of Proposition \ref{PROP:SZ}, we have for $t\geq \eps|\log\eps|$ that
\begin{align*}
|A'|\les \frac{\eps}{|2\eps-i(c_2-c_1)t|}\les |\log\eps|^{-\frac 12}\ll 1. 
\end{align*}
Therefore, for $t\geq \eps|\log\eps|$, we have indeed that
\begin{align*}
\| V_1^{\eps}(t)- V_2^\eps(t)\|_{L^2_+}\sim 1.
\end{align*}

\end{proof}

\begin{remark}\label{REM:UC SZ}
{\rm 
In the proof of Corollary \ref{COR:UC.SZ},
instead of using slightly different approaches for the
cases $s\in (0,\frac 12)$ and $s=0$, 
one can choose to work, for all $s\in [0,\frac 12)$, 
with the following traveling waves of the Szeg\H{o} equation:
\begin{align*}
V_j^\eps(t,x):=\frac{\al_je^{-i\o_jt}}{x-c_jt+ip}, \quad \text{where} \quad p:=\eps, \qquad \al_1:=\eps^{s+\frac 12}, \qquad \al_2:=\eps^{s+\frac 12}(1+|\log\eps|^{-\frac 12}).
\end{align*} 
For $s\in (0,\frac 12)$, however, we preferred to use the basic construction in Proposition \ref{PROP:SZ} 
together with a scaling argument, as a preamble to our proof of the 
failure of local uniform continuity for NHW in $H^s(\R)$, $s\in (0,\frac 12)$, in Proposition \ref{PROP:UC}.
}
\end{remark}

Next, we recall the 
result from \cite{Poc2} on long time approximation of solutions of the NHW equation \eqref{HW}
by solutions of the Szeg\H{o} model \eqref{SZ}. 
\begin{proposition}[\cite{Poc2}]\label{PROP:approx}
Let $0<\eps\ll1$, $\delta>0$ sufficiently small, $s>\frac 12$, and $f\in H^s_+(\R)$.
Let $\mu\in\{-1,+1\}$ and $u_\mu\in C(\R; H^s)$ be the solution of \eqref{HW} 
with initial data $u_\mu(0)=\eps f$. 

Let $V\in C(\R;H^s_+)$ be the solution of the Szeg\H{o} equation \eqref{SZ}
with the same initial data $V(0)=\eps f$. 
Assume that $\|V(t)\|_{H^s}\leq C\eps$ for all $0\leq t\leq \frac{\delta}{\eps^2}|\log\e|$. 
Then, for any $0\leq t\leq \frac{\delta}{\eps^2}|\log\e|$, the following holds:
\[\|u_\mu(t,\cdot)-V(t,\cdot -t)\|_{H^s}\leq C_\ast \eps^{2-C_0\delta},\]
where $C_0>0$ is an absolute constant and $C_{\ast}=C_\ast (\|f\|_{H^{\frac 12}})$.
\end{proposition}

With this approximation result at hand, we are now ready to
state and prove the failure of uniform continuity 
of the solution map of \eqref{HW} on bounded sets of $H^s(\R)$, $s\in (0, \frac 12)$.
\begin{proposition}[Failure of local uniform continuity for NHW in $H^s(\R)$, $s\in (0, \frac 12)$]\label{PROP:UC}

Let $s\in (0,\frac 12)$ and $\mu\in\{-1,+1\}$. 
Given $0<\eps\ll 1$, there exist global solutions 
$u_1^\eps$ and $u_2^\eps$ of \eqref{HW} such that 
$\|u_1^\e(0)\|_{H^s}\les 1$, $\|u_2^\e(0)\|_{H^s}\les 1$,
\begin{equation*}
\lim_{\eps\to 0}\| u_1^\eps(0)- u_2^\eps(0)\|_{H^s}=0,
\end{equation*}
and
\begin{equation*}
\liminf_{\eps\to 0}\| u_1^\eps- u_2^\eps\|_{L^\infty([0,T];H^s)}\gtrsim 1 \quad \text { for all } \quad T>0.
\end{equation*}
\end{proposition} 

\begin{proof}[Proof of Proposition \ref{PROP:UC}]
We set $\tilde v_j^\eps(t,x):=e^{-i|D|t}\tilde V_j^\eps(t,x)=\tilde V_j^\eps(t,x-t)$, $j=1,2$,
where $\tilde V_j^\eps$ are the traveling waves of the Szeg\H{o} equation
introduced in \eqref{V_j} and \eqref{p_al}. 
By Proposition \ref{PROP:SZ}, it follows that
\begin{equation}\label{SZ1bis}
\|\tilde v_1^\eps(0)-\tilde v_2^\eps(0)\|_{H^s}\sim \eps|\log\eps|^{-\frac 12}
\end{equation}
and
\begin{equation}\label{SZ2bis}
\|\tilde v_1^\eps(t)-\tilde v_2^\eps(t)\|_{H^s}\ges \eps
\end{equation}
for all $t\geq \frac{\delta}{\eps^2}|\log\eps|$,
where $\delta>0$ is a small real number to be chosen later.

Next, we denote by $\tilde u_j^\eps$, $j=1,2$,
the smooth solutions of \eqref{HW}
with initial data
\[\tilde u_j^\eps(0,x):=\tilde v_j^\eps(0,x)=\frac{\al_j}{x+i},\]
where $\al_j$ are as in \eqref{p_al}.
Note that the initial conditions $\tilde u_j^\eps(0)$
are sufficiently smooth and small
to guarantee 
the global existence of the corresponding solutions 
of NHW in both the defocusing and focusing cases.

By Lemma \ref{Fourier} 
we have
$\|\tilde u_j^\eps(0)\|_{H^s}\les\eps$ and $\|\tilde v_j^\eps(t)\|_{H^s}\sim \eps$ for all $t\in\R$.
Thus, applying Proposition \ref{PROP:approx},
it then follows that for all $0\leq t\leq \frac{\delta}{\eps^2}|\log\eps|$ we have
\begin{align}\label{u-v}
\|\tilde u_j^\eps(t)-\tilde v_j^\eps(t)\|_{H^s}\lesssim \eps^{2-C_0\delta},
\end{align}
where $C_0>0$ is an absolute constant.
Combining \eqref{SZ1bis}, \eqref{SZ2bis}, and \eqref{u-v},
we then obtain that:
\begin{align*}
\|\tilde u_1^\eps(0)-\tilde u_2^\eps(0)\|_{H^s}&\sim \eps|\log\eps|^{-\frac 12}\\
\left\|\left(\tilde u_1^\eps-\tilde u_2^\eps\right)\left(\frac{\delta}{2\eps^2}|\log\eps|\right)\right\|_{H^s}&\ges \eps,
\end{align*}
provided that $\delta$ is so small that $2-C_0\dl> 1$. 
We set $u_j^\eps:=(\tilde u_j^\eps)_{\lambda}$ with $\ld=\eps^{-\frac 1s}$,
and observe that $u_j^\eps$ is also a solution of NHW.
Then, we obtain as in the proof of Corollary \ref{COR:UC.SZ}
that
\begin{align*}
&\|u_1^\eps(0)\|_{H^s}+\|u_2^\eps(0)\|_{H^s}\les 1,\\
&\|u_1^\eps(0)- u_2^\eps(0)\|_{H^s}\sim |\log\eps|^{-\frac 12}\ll 1,\\
&\| u_1^\eps(T_\eps)- u_2^\eps(T_\eps)\|_{H^s}\ges 1,
\end{align*}
where $T_\eps:=\dl\eps^{\frac1s -2}|\log\eps|\ll 1$.
This concludes the proof of the proposition.
\end{proof}

\begin{remark}\label{REM:GTV}
{\rm
Setting $v_j^\eps:=(\tilde v_j^\eps)_\ld$
with $\ld=\eps^{-\frac 1s}$ (where $\tilde v_j^\eps$ were defined in the proof of Proposition \ref{PROP:UC}) and $\tilde\eps:=\eps^{\frac 1s}$, we have that
\[v_j^\eps(t,x)=\frac{\tilde \al_j e^{-i\tilde \omega_j t}}{x-(1+c_j)t+i\tilde\eps},\]
where
\begin{align*}
\tilde \alpha_1=\tilde\eps^{s+\frac 12}, \qquad \tilde \alpha_2=\tilde\eps^{s+\frac 12}(1+C(s)|\log\tilde\eps|^{-\frac 12}), \qquad \tilde\omega_j=\omega_j \tilde\eps^{-1}, \qquad j=1,2. 
\end{align*}
Then, we can reformulate \eqref{u-v} as
\[\|u_j^\eps(t)-v_j^\eps(t)\|_{H^s}\lesssim \tilde\eps^{s(1-C_0\delta)}\]
for all $0\leq t\les \delta\tilde\eps^{1-2s}|\log\tilde\eps|$ and $\delta>0$ sufficiently small. 

We remark that $v_j^\eps$ are 
(translated) traveling waves of the Szeg\H{o} equation on $\R$
analogous to the traveling waves considered in \cite{GTV} for the Szeg\H{o} equation on $\T$. 
In \cite[Proposition 3.1]{GTV}, the difference
$\|u_j^\eps-v_j^\eps\|_{H^s}$ is bounded above by $\tilde\eps^{s-\frac 14}$.
This explains why the failure of local uniform continuity 
for periodic NHW obtained in \cite{GTV} occurs in $H^s(\T)$ with the restriction on the regularity $s\in (\frac 14,\frac 12)$.
In our context, the upper bound $\tilde\eps^{s(1-C_0\delta)}$ allows 
for the wider regularity range $s\in (0,\frac 12)$. 
}
\end{remark}

\begin{remark}\label{REM:UC}
{\rm 
The scaling used in the proof of Proposition
\ref{PROP:UC},
$u_j^\eps:=(\tilde u_j^\eps)_{\lambda}$ with $\ld=\eps^{-\frac 1s}$,
is only defined for $s\neq 0$.  
Therefore,
even though the solution map of the Szeg\H{o} equation on $\R$
fails to be locally uniformly continuous
in $L^2(\R)$, as proved in Corollary \ref{COR:UC.SZ} above,
one cannot use the approximation/scaling argument in Proposition \ref{PROP:UC}
to deduce the same behavior for the solution map of NHW on $L^2(\R)$.
}
\end{remark}

\section{Failure of local uniform continuity of the solution map of the focusing NHW in $L^2(\R)$}

In this section, we prove Theorem
\ref{THM:main} (ii). 
More precisely, we consider the focusing half-wave equation \eqref{HW} (with $\mu=-1$) and show that
its solution map fails to be uniformly continuous on bounded sets of $L^2(\R)$. 
As noted above in Remark \ref{REM:UC},
the approximation/scaling argument from Proposition \ref{PROP:UC}
can no longer be used in the case of $L^2(\R)$.
Consequently, we consider a different approach. 
Namely, instead of working with solutions of NHW that can be approximated by 
traveling waves of the Szeg\H{o} equation, 
we work directly with traveling waves of the focusing NHW.

\begin{proposition}[Failure of local uniform continuity for focusing NHW in $L^2(\R)$]\label{PROP:UCL2}
Given $0<\eps\ll 1$, there exist global solutions $u_1^\eps$ and $u_2^\eps$ 
of the focusing half-wave equation \eqref{HW} with $\mu=-1$ such that
$\|u_1^\e(0)\|_{L^2}\les 1$, $\|u_2^\e(0)\|_{L^2}\les 1$,
\begin{equation*}
\lim_{\eps\to 0}\| u_1^\eps(0)- u_2^\eps(0)\|_{L^2}=0,
\end{equation*}
and
\begin{equation*}
\liminf_{\eps\to 0}\| u_1^\eps- u_2^\eps\|_{L^\infty([0,T];L^2)}\gtrsim 1 \quad \text { for all } \quad T>0.
\end{equation*}
\end{proposition} 

The solutions $u_1^\eps$,  $u_2^\eps$,
that we use to prove Proposition \ref{PROP:UCL2}
are conveniently rescaled versions of traveling waves of the focusing NHW.
In \cite{KLR}, Krieger, Lenzmann, and Rapha\"el showed 
that for any $-1<\beta<1$, the focusing cubic half-wave equation on $\R$ possesses 
a traveling wave solution $u_\b(t,x):=Q_\b\left(\frac{x-\b t}{1-\b}\right)e^{it}$,
where $Q_\beta\in H^{\frac 12}(\R)$ satisfies
\begin{align*}
\frac{|D|-\b D}{1-\b}Q_\b+Q_\b=|Q_\b|^2Q_\b.
\end{align*}
The crucial element in the proof of Proposition \ref{PROP:UCL2} is the use of 
certain properties of $Q_\b$ that we recall below from \cite{GLPR}.
\begin{lemma}[\cite{GLPR}, Properties of $Q_\b$]\label{LEM:Qb}
There exists $0<\b_\ast<1$ such that for all $\b,\tilde\b \in (\b_\ast,1)$,
the following hold:
\begin{align}
\|Q_\b- Q_{\tilde \b}\|_{H^{\frac 12}}&\leq C \frac{|\b-\tilde \b|}{\min (1-\b, 1-\tilde \b)}\label{Qb1}\\
\|x\pa_x Q_\b\|_{L^2}&\leq C\label{Qb2}\\
|Q_\b(x)|&\leq  \frac{C}{\jb{x}(1+(1-\b)\jb{x})},\label{Qb3}
\end{align}
where $C>0$ is an absolute constant {\rm(}independent of $\b$, $\tilde\b${\rm)} and $\jb{x}:=\sqrt{1+x^2}$.

Moreover, 
given $\b\in (\b_\ast, 1)$,
there exist constants $x(\b)\in \R$ and $\gamma\in\T$
such that, up to a subsequence,
\begin{align}
\|Q_\b(x-x(\b))-e^{i\gamma}Q^+(x)\|_{H^{\frac 12}}\leq C_1(1-\b)^{\frac 18},\label{Qb4}
\end{align}
where $Q^+(x)=\frac{2}{2x+i}$ and $C_1>2$ is an absolute constant. 
In particular, 
$Q_{\b}(\cdot-x(\b))\to e^{i\gamma}Q^+$ in $H^{\frac 12}(\R)$ as $\beta\to 1$.
\end{lemma}

\noi
The proof of Lemma \ref{LEM:Qb}
is lengthy and does not constitute the object of this paper. 
Therefore, we decided to omit it here and we refer the readers to \cite{GLPR} for details.

The following lemma is another useful tool in the proof of Proposition \ref{PROP:UCL2}.
\begin{lemma}[\cite{GLPR}, Auxiliary lemma]\label{LEM:auxiliary}
Let $\b\in (0,1)$. Then, the following holds:
\begin{align*}
\int_{\R}\frac{1}{\jb{x-y}(1+(1-\b)\jb{x-y})}\cdot \frac{1}{\jb{y}(1+(1-\b)\jb{y})}dy\lesssim \frac{|\log(1-\b)|}{\jb{x}(1+(1-\b)\jb{x})}.
\end{align*}
\end{lemma}

The proof of Lemma \ref{LEM:auxiliary} is elementary and consists 
of analyzing separately the following three regions of integration:
(i) $|y|\geq 2|x|$, (ii) $|y|\leq 2|x|$ and $|x-y|\geq \frac{|x|}{2}$, (iii) 
$|y|\leq 2|x|$ and $|x-y|\leq \frac{|x|}{2}$. Details can be found in the appendix of \cite{GLPR}.

\begin{proof}[Proof of Proposition \ref{PROP:UCL2}]
We choose $0<\eps\ll 1$, $c_0\in \left(0, 1-\max\left(\b_\ast, 1-\big(\frac{\pi}{2C_1^2}\big)^4\right)\right)$,
where $C_1$ is as in \eqref{Qb4}, and $\b_1$ and $\b_2$ such that
\begin{align}\label{betas}
0<\max\left(\b_\ast, 1-\big(\frac{\pi}{2C_1^2}\big)^4\right)<\b_1<\b_2<1-c_0, \quad \eps^{\frac 43}(1-\b_2)<\b_2-\b_1<\eps (1-\b_2).
\end{align}
In other words, the speeds $\b_1$ and $\b_2$ are sufficiently close to $1$, but away from $1$, 
and $\eps$-close to each other. 

We start by 
estimating the difference of the initial data of the solutions $u_{\b_1}$ and $u_{\b_2}$.
By \eqref{Qb1}, \eqref{Qb2}, and \eqref{betas}, it follows that
\begin{align}\label{A-1}
\|u_{\beta_1}(0,x)-&u_{\beta_2}(0,x)\|_{L^2}
=\left\|Q_{\beta_1}\left(\frac{x}{1-\b_1}\right)-Q_{\beta_2}\left(\frac{x}{1-\b_2}\right)\right\|_{L^2}\notag\\
&=\sqrt{1-\b_1}\left\|Q_{\b_1}(x)-Q_{\b_2}\left(\frac{1-\b_1}{1-\b
_2}x\right)\right\|_{L^2}\notag\\
&\leq\sqrt{1-\b_1}\left(\|Q_{\b_1}(x)-Q_{\b_2}(x)\|_{L^2}+\left\|Q_{\b_2}(x)-Q_{\b_2}\left(\frac{1-\b_1}{1-\b_2}x\right)\right\|_{L^2}\right)\notag\\
&\les\sqrt{1-\b_1}\frac{\b_2-\b_1}{1-\b_2}\left(1+\sup_{c\in [1,\frac{1-\b_1}{1-\b_2}]}\|x(\pa_x Q_\beta)(cx)\|_{L^2}\right)\les\sqrt{1-\b_1}\frac{\b_2-\b_1}{1-\b_2}\notag\\
&\les \frac{\b_2-\b_1}{\sqrt{1-\b_2}}+\frac{(\b_2-\b_1)^{\frac 32}}{1-\b_2}\les \eps \sqrt{1-\b_2}\les \eps.
\end{align}

\noi
Next, we estimate the difference of the solutions at time $t$.
\begin{align}
&\|u_{\b_1}(t,x)-u_{\b_2}(t,x)\|_{L^2}^2
=\left\|Q_{\b_1}\left(\frac{x-\b_1 t}{1-\b_1}\right)-Q_{\b_2}\left(\frac{x-\b_2 t}{1-\b_2}\right)\right\|_{L^2}^2\notag\\
&=\left\|Q_{\b_1}\left(\frac{x-\b_1 t}{1-\b_1}\right)\right\|_{L^2}^2
+\left\|Q_{\b_2}\left(\frac{x-\b_2 t}{1-\b_2}\right)\right\|_{L^2}^2
-2\Re\int_{\R} \cj{Q_{\b_1}\left(\frac{x-\b_1 t}{1-\b_1}\right)}Q_{\b_2}\left(\frac{x-\b_2 t}{1-\b_2}\right)dx\notag\\
&=(1-\b_1)\|Q_{\b_1}\|_{L^2}^2+(1-\b_2)\|Q_{\b_2}\|_{L^2}^2
-B,\label{A0}
\end{align}
where
\begin{align*}
B:=2(1-\b_1)\Re\int_{\R}\cj{Q_{\b_1}(y_1)}Q_{\b_2}\left(\frac{1-\b_1}{1-\b_2}y_1+\frac{\b_1-\b_2}{1-\b_2}t\right)dy_1.
\end{align*}
We first estimate $|B|$. By \eqref{Qb3} and setting $y_2:=\frac{1-\b_1}{1-\b_2}y_1+\frac{\b_1-\b_2}{1-\b_2}t$, we have
\begin{align}\label{A1}
|B|\les (1-\b_1)\int_{\R} \frac{1}{\jb{y_1}(1+(1-\b_1)\jb{y_1})}\cdot \frac{1}{\jb{y_2}(1+(1-\b_2)\jb{y_2})}dy_1.
\end{align}
Using $\b_2-\b_1<\eps(1-\b_2)<\eps (1-\b_1)$, we notice that $\frac{1-\b_2}{1-\b_1}=1-\frac{\b_2-\b_1}{1-\b_1}>1-\eps$ and thus,
\begin{align}
\frac{1}{\jb{y_2}}\sim \frac{1}{1+|y_2|}
&\sim \frac{1-\b_2}{1-\b_1}\cdot \frac{1}{\frac{1-\b_2}{1-\b_1}+|y_1+\frac{\b_1-\b_2}{1-\b_1}t|}
\les \frac{1-\b_2}{1-\b_1}\cdot \frac{1}{1-\e+|y_1+\frac{\b_1-\b_2}{1-\b_1}t|}\notag\\
&\les  \frac{1-\b_2}{1-\b_1}\cdot \frac{1}{\jb{y_1+\frac{\b_1-\b_2}{1-\b_1}t}}.\label{A2}
\end{align}
Similarly, we have
\begin{align}
\frac{1}{1+(1-\b_2)\jb{y_2}}
&\sim \frac{1}{1+(1-\b_1)\left(\frac{1-\b_2}{1-\b_1}+|y_1+\frac{\b_1-\b_2}{1-\b_1}t|\right)}
\les \frac{1}{1+(1-\b_1)\jb{y_1+\frac{\b_1-\b_2}{1-\b_1}t}}.\label{A3}
\end{align}
Therefore, by \eqref{A1}, \eqref{A2}, \eqref{A3}, and using Lemma \ref{LEM:auxiliary},
we obtain that
\begin{align*}
|B|&\les (1-\b_2)\int_{\R} \frac{1}{\jb{y_1}(1+(1-\b_1)\jb{y_1})}\cdot \frac{1}{\jb{y_1+\frac{\b_1-\b_2}{1-\b_1}t}\left(1+(1-\b_1)\jb{y_1+\frac{\b_1-\b_2}{1-\b_1}t}\right)}dy_1\\
&\les \frac{(1-\b_2)|\log (1-\b_1)|}{\jb{\frac{\b_1-\b_2}{1-\b_1}t}\left(1+(1-\b_1)\jb{\frac{\b_1-\b_2}{1-\b_1}t}\right)}
\lesssim \frac{(1-\b_1)(1-\b_2)|\log(1-\b_1)|}{\left((\b_2-\b_1)t\right)^2}.
\end{align*}
By choosing $t\geq \frac{1}{\sqrt{\eps}(\b_2-\b_1)}$ and using $(1-\b_1)|\log(1-\b_1)|\les 1$,
it follows that
\begin{align}
|B|&\ll \eps(1-\b_2)<\eps(1-\b_1).\label{A4}
\end{align}
On the other hand, by \eqref{Qb4}, using $\|Q^+\|_{L^2}=\sqrt{2\pi}$ and $1-\b_2<1-\b_1<\big(\frac{\pi}{2C_1^2}\big)^4$,
we have
\begin{align}\label{A5}
\|Q_{\b_j}\|_{L^2}\geq \|Q^+\|_{L^2}-\|Q_{\b_j}(\cdot-x(\b_j))-e^{i\gamma}Q^+\|_{L^2}
\geq \sqrt{2\pi}-C_1(1-\b_1)^{\frac 18}\geq \sqrt{\frac{\pi}{2}}
\end{align}
for $j=1,2$. 
Then, combining \eqref{A0}, \eqref{A4}, and \eqref{A5},
we obtain for $t\geq \frac{1}{\sqrt{\eps}(\b_2-\b_1)}$ that
\begin{align}\label{A6}
&\|u_{\b_1}(t)-u_{\b_2}(t)\|_{L^2}\ges \sqrt{1-\b_1}\ges \sqrt{c_0}.
\end{align}

Finally, by considering the rescaled variants of $u_{\b_j}$
$u_j^\eps(t,x):=\eps^{-1}u_{\b_j}\left(\e^{-2}t, \e^{-2}x\right)$
and using \eqref{Qb4}, \eqref{A-1}, and \eqref{A6}, we obtain that 
$\|u_j^\eps(0)\|_{L^2}=\sqrt{1-\b_j}\|Q_{\b_j}\|_{L^2}\les \sqrt{1-\b_j} \les \sqrt{1-\b_\ast}$
for $j=1,2$,
$\|u_1^\e(0)-u_2^\eps(0)\|_{L^2}\lesssim \eps$,
and
\begin{align*}
\|u_1^\e(t)-u_2^\eps(t)\|_{L^2}\ges 1
\end{align*}
for all $t\geq \frac{\eps^{\frac 32}}{\b_2-\b_1}$, and in particular for all $t\ges \eps^{\frac 16}$ (since $\b_2-\b_1>\eps^{\frac 43}(1-\b_2)$). 
This completes the proof. 
\end{proof}

\begin{remark}\label{defocNHW_L2}
{\rm 
(i). For the defocusing nonlinear half-wave equation \eqref{HW} with $\mu=1$,
traveling waves are not available and therefore 
the above approach cannot be used to prove failure of local uniform continuity 
in $L^2(\R)$. 

For the defocusing NHW, it is natural to attempt to use the approach of 
Christ, Colliander, and Tao from \cite{CCTarxiv}. 
Namely, one would like to consider two solutions
$u_j(t,x)=\phi_j(t,\nu x)$, 
$j=1,2$, $0<\nu\ll1$, where $\phi_j$ 
satisfy:
\begin{align*}
\begin{cases}
i\pa_t \phi_j-\nu |D|\phi_j=|\phi_j|^2\phi_j\\
\phi_j(0,x)=a_jw(x),
\end{cases}
\end{align*}
with $w$ a fixed Schwartz function, $a_1,a_2 \in [\frac 12,2]$, and $|a_1-a_2|\ll 1$. 
One can then approximate $\phi_j$, for a long time $0\leq t\leq c|\log\nu|^c$,
by the solution $\phi_j^{(0)}$ of dispersionless NLS,
$i\partial_t \phi^{(0)}_j=\big|\phi^{(0)}_j\big|^2\phi^{(0)}_j$,
with the same initial data $\phi^{(0)}_j(0,x)=\phi_j(0,x)$. 
For $\phi^{(0)}_j$ we have an explicit formula
$\phi^{(0)}_j(x)=a_jw(x)e^{ia_j^2t|w(x)|^2}$.
Therefore, one can show without difficulty that 
$\|\phi^{(0)}_1(0)-\phi_2^{(0)}(0)\|_{L^2}\les |a_1-a_2|\ll 1$ and
$\|\phi^{(0)}_1(t)-\phi_2^{(0)}(t)\|_{L^2}\ges 1$
provided that $t\gg \frac{1}{|a_1-a_2|}$.
However, the time on which the approximation of $\phi_j$
by $\phi_j^{(0)}$ holds,
does not seem to be sufficiently long to obtain the same 
statement for $u_1-u_2$. 
More precisely, one only obtains
$\frac{\|u_1(t)-u_2(t)\|_{L^2}}{\|u_1(0)-u_2(0)\|_{L^2}}\ges \frac{1}{|a_1-a_2|}\gg 1$
with $\|u_1(0)-u_2(0)\|_{L^2}\gg1$. 

In \cite{CCTarxiv}, in the case of the nonlinear Schr\"odinger equation 
below the scaling critical regularity, 
this issue is addressed by using the scaling and Galilean symmetries of NLS and,
as a result, one
obtains indeed failure of local uniform continuity.  
In our context, however, we are at the scaling critical regularity and, therefore, the scaling
is not useful. Moreover, NHW does not have a Galilean symmetry, 
nor a Lorenz symmetry. Therefore, unless a new invariance is found for NHW,
this approach does not seem viable. 
We note that in \cite{BGT}, 
the use of scaling was avoided for super-quintic NLS on a three dimensional manifold
by working with highly localized initial data.   
The strategy of \cite{BGT} can be applied to show failure of local uniform continuity 
for a defocusing super-cubic half-wave equation in $L^2(\R)$,
but not for the cubic NHW that we consider here.  

\smallskip
\noi
(ii). As we have seen in the proof of Proposition \ref{PROP:UC},
the approximation by 
the Szeg\H{o} equation does not seem sufficient to decide on
the failure of local uniform continuity for the defocusing NHW in $L^2(\R)$ (at least, not 
for the examples of solutions
considered in Proposition \ref{PROP:UC}).
It would be interesting to find a better approximation of the defocusing NHW 
that might provide us with a more accurate understanding of the dynamics.

}
\end{remark}


\section{Failure of $C^3$-smoothness of the solution map of NHW in $L^2(\R)$}

In the previous section, we discussed
the failure of local uniform continuity of the solution map of NHW in $L^2(\R)$. 
In Proposition \ref{PROP:UCL2},
we showed this for the focusing NHW. 
It remains, however, an open question
in the case of the defocusing NHW. 
In this section, we prove a weaker form of ill-posedness in $L^2(\R)$
for both the defocusing and focusing 
NHW.
Namely,
we show that the solution map of NHW
fails to be $C^3$-smooth in $L^2(\R)$ (assuming that it is well defined as a mapping on $L^2(\R)$).
\begin{proposition}[Failure of $C^3$-smoothness of the solution map of NHW in $L^2(\R)$]\label{PROP:C3}
Let $\mu \in \{-1,1\}$ and fix $0<t\leq1$. 
Denote the solution map of \eqref{HW} by $\Phi(t): u_0\mapsto u(t)$. 
Assuming that $\Phi(t)$ is well-defined as a map acting on $L^2(\R)$,
it then follows that $\Phi(t)$ is not $C^3$-smooth at $u_0=0$ in $L^2(\R)$.

\end{proposition}

\begin{proof}
If the solution map $\Phi(t)$ were to be $C^3$-smooth at zero in $L^2(\R)$,
then there would exist $C>0$ such that for all $f\in L^2(\R)$: 
\[\left\|\frac{d^3 \Phi(t)(\delta f)}{d\delta^3}\Big|_{\delta=0}\right\|_{L^2}\leq C \|f\|_{L^2}^3.\]
In the following we show that such an estimate cannot hold with a constant
independent of $f\in L^2(\R)$. 

By Duhamel's formula, we have
\begin{align*}
\Phi(t)(\delta f)=\delta e^{-it|D|}f-i\int_0^t e^{-i(t-t')|D|}|\Phi(t')(\delta f)|^2\Phi(t')(\delta f)dt'.
\end{align*}
In turn, $\frac{d\Phi(t)(\delta f)}{d\delta}\Big|_{\delta=0}=e^{-it|D|}f$,
$\frac{d^2 \Phi(t)(\delta f)}{d\delta^2}\Big|_{\delta=0}=0$, and
\begin{align*}
\frac{d^3 \Phi(t)(\delta f)}{d\delta^3}\Big|_{\delta=0}=-6i\int_0^te^{-i(t-t')|D|}(|e^{-it'|D|}f|^2e^{-it'|D|}f)(x)dt'.
\end{align*}
To prove the failure of $C^3$-smoothness of $\Phi(t)$ at zero in $L^2(\R)$,
we show that
\begin{align}\label{C3}
\left\|\int_0^te^{-i(t-t')|D|}(|e^{-it'|D|}f|^2e^{-it'|D|}f)(x)dt'\right\|_{L^2}\leq C\|f\|_{L^2}^3
\end{align}
cannot hold uniformly in $f\in L^2(\R)$.

Consider $f_\eps(x)=\frac{1}{x+i\eps}\in L^2_+(\R)$. 
By Lemma \ref{Fourier}, 
we have that 
\begin{equation}\label{C30}
\|f_\eps\|_{L^2}=\frac{1}{\sqrt{2\eps}}.
\end{equation}
Then, noticing that
$e^{-it|D|}\Pi_+g(x)=\Pi_+g (x-t)$, $e^{-it|D|}\Pi_-g(x)=\Pi_-g (x+t)$, for all $g\in L^2(\R)$, along with $\Pi_+f_\eps=f_\eps$, 
it follows that
\begin{align}
\int_0^te^{-i(t-t')|D|}&(|e^{-it'|D|}f_\e|^2e^{-it'|D|}f_\e)(x)dt'
=\int_0^t e^{-i(t-t')|D|}(|f_\e|^2f_\e)(x-t')dt'\notag\\
&=\int_0^t \left(\Pi_+(|f_\e|^2f_\e)(x-t) + \Pi_-(|f_\e|^2f_\e)(x+t-2t')\right)dt'\notag\\
&=t\Pi_+(|f_\e|^2f_\e)(x-t) + \int_0^t\Pi_-(|f_\e|^2f_\e)(x+t-2t')dt'.\label{C31}
\end{align}
Decomposing into simple fractions gives
\begin{align*}
|f_\eps|^2f_\eps(x)=\frac{1}{4\eps^2}\cdot \frac{1}{x+i\eps}-\frac{1}{2i\eps}\cdot \frac{1}{(x+i\eps)^2}
-\frac{1}{4\eps^2}\cdot \frac{1}{x-i\eps},
\end{align*}
and observe that the first two terms are supported on non-negative frequencies,
while the last term is supported on negative frequencies. Thus,
\begin{align*}
\Pi_+(|f_\eps|^2f_\eps)(x)=\frac{1}{4\eps^2}f_\eps(x)+\frac{1}{2i\eps}\pa_x f_\eps(x), \qquad \Pi_-(|f_\e|^2f_\e)(x)=-\frac{1}{4\e^2}\cdot \frac{1}{x-i\e}.
\end{align*}
Therefore, using again Lemma \ref{Fourier},
we obtain that
\begin{equation}\label{C32}
\|t\Pi_+(|f_\eps|^2f_\eps)(x-t)\|_{L^2}=\frac{t\sqrt{5\pi}}{4\eps^2\sqrt{\eps}}.
\end{equation}
and
\begin{align*}
\mathcal{F}\left(\int_0^t\frac{dt'}{x+t-2t'-i\eps}dt' \right)(\xi)
=2\pi e^{it\xi}\frac{1-e^{-2it\xi}}{2\xi}e^{\eps\xi}\pmb{1}_{\xi\leq 0}.
\end{align*}
Using $|\frac{1-e^{-2it\xi}}{2\xi}|\leq t$, we then have that
\begin{align}
\left\|\int_0^t\Pi_-(|f_\e|^2f_\e)(x+t-2t')dt'\right\|_{L^2_x}
&=\frac{1}{4\eps^2}\left\|\mathcal{F}\left(\int_0^t\frac{dt'}{x+t-2t'-i\eps}\right)(\xi)\right\|_{L^2_\xi}\notag\\
\lesssim \frac{1}{\eps^2}\left(t^2+\int_{-\infty}^{-1}\frac{1}{\xi^2}d\xi\right)^{\frac 12}
\lesssim \frac{1}{\e^2}. \label{C33}
\end{align}
Combining \eqref{C31}, \eqref{C32}, and \eqref{C33}, and recalling that $0<t\leq 1$ is fixed,
we obtain that
\begin{align}\label{C34}
\Big\| \int_0^te^{-i(t-t')|D|}&(|e^{-it'|D|}f_\e|^2e^{-it'|D|}f_\e)(x)dt' \Big\|_{L^2_x}
\gtrsim \frac{t}{\e^2\sqrt{\e}}
\end{align}
for $\eps$ sufficiently small.
By making $\eps$ tend to zero,
it follows by \eqref{C30} and \eqref{C34}
that there is no constant $C$ for which \eqref{C3} holds.
Therefore, indeed, $\Phi(t)$ is not
$C^3$-smooth in $L^2(\R)$.

\end{proof}

\begin{remark}\label{REM:C3 SZ}
{\rm
(i) The proof of Proposition \ref{PROP:C3}
can be easily adapted to show the failure of $C^3$-smoothness of the solution map 
of NHW in $H^s(\R)$ for any $s\in [0,\frac 12)$.

\smallskip
\noi
(ii)  A simplified variant of the proof of Proposition \ref{PROP:C3} yields the failure of $C^3$-smoothness of the solution map of the Szeg\H{o} equation
in $H^s_+(\R)$, $s\in [0,\frac 12)$. 
}
\end{remark}

%
%
%
%
%

\section{Norm inflation property for NHW in $H^s(\R)$, $s<0$}
\label{SEC:inflation}

This section is dedicated to the proof of Theorem \ref{THM:main} (iv).
Namely, we show that a $H^s$-norm inflation phenomenon occurs
for certain solutions of NHW when $s<0$.

The analysis in this section follows closely 
an argument developed by Kishimoto \cite{Kishimoto1}
in the context of the one-dimensional periodic
cubic nonlinear Schr\"odinger equation (see also \cite{Iwabuchi_Ogawa}).
An important tool is the use of an algebra contained in $L^2(\R)$.  
We choose this algebra to be the following (scaled) modulation space.
\begin{definition}
Given $A\geq 1$, let $I_A:=[-\frac A2, \frac A2)$.
We define $M_A(\R)$ to be the completion of $C^\infty(\R)$
with respect to the norm:
\begin{equation}
\|f\|_{M_A}:=\sum_{k\in A\Z}\|\hat{f}\|_{L^2(k+I_A)}.
\end{equation}
\end{definition}
\noi

Modulation spaces were introduced by Feichtinger in \cite{Feichtinger}
and the basic theory of these spaces was established in \cite{FG1, FG2}.
See also \cite{Benyi_Okoudjou} for an application
of modulation spaces to the local well-posedness theory of nonlinear dispersive PDEs.
In the present paper, we only use
the following two properties of the modulation space $M_A$.
\begin{lemma}[Properties of the modulation space $M_A$]\label{algebra}
Let $A\geq 1$.

\noi
{\rm (i)} There exists an absolute constant $C>0$ such that 
$\|f\|_{L^2}\leq C\|f\|_{M_A}$ for all $f\in M_A$. 

\smallskip
\noi
{\rm (ii)} There exists $C_2>0$ absolute constant such that for any $f,g\in M_A$
the following holds:
\begin{equation}
\|fg\|_{M_A}\leq C_2 A^{\frac 12}\|f\|_{M_A}\|g\|_{M_A}.
\end{equation}
\end{lemma}

The algebra property (ii) in Lemma \ref{algebra}
allows one to easily show that NHW
is locally well-posed in $M_A$. 
Before stating this local well-posedness result,
we set the following notations 
for $\phi\in M_A$:
\begin{align}
U_1[\phi](t):&=e^{-it|D|}\phi\notag\\
U_k[\phi](t):&=-i\mu\sum_{\substack{k_1,k_2, k_3\geq 1\\ k_1+k_2+k_3=k}}
\int_0^te^{-i(t-\tau)|D|}\big(U_{k_1}[\phi]\cj{U_{k_2}[\phi]}U_{k_3}[\phi]\big)(\tau)d\tau.\label{Uk}
\end{align}
Here, $U_k[\phi]$ is the sum of all the terms 
that contain exactly $k$ factors $e^{-i\tau|D|}\phi$ 
in the Picard iteration process of constructing a solution of \eqref{HW} with initial
condition $\phi$.
Note also that $U_k[\phi]\equiv 0$ for all even $k$.
\begin{lemma}[Local well-posedness of NHW in $M_A$]\label{LWP}
Let $\mu\in\{-1,1\}$, $A\geq 1$, and $\phi\in M_A$.
There exists a unique solution $u\in C([0,T_\ast]; M_A)$ of \eqref{HW}, where
$T_\ast=C_3A^{-1}\|\phi\|_{M_A}^{-2}$ and $C_3>0$ is an absolute constant.
Moreover,
\begin{equation}\label{expansion}
u=\sum_{k=1}^\infty U_k[\phi],
\end{equation}
where the series converges absolutely in $C([0,T_\ast];M_A)$.
\end{lemma}

\begin{proof}
The proof is via a standard fixed point argument. 
We consider the operator
\begin{equation*}
\Gamma u(t):=e^{-it|D|}\phi-i\mu\int_0^te^{-i(t-\tau)|D|}|u|^2u(\tau)d\tau.
\end{equation*}
By Lemma \ref{algebra}, we have for $u$ in the ball $B(0,2\|\phi\|_{M_A})$ in $C([0,T_\ast]; M_A)$ that
\begin{align*}
\sup_{t\in [0,T_\ast ]}\|\Gamma u(t)\|_{M_A}
\leq \|\phi\|_{M_A}+CT_\ast A\|u\|_{M_A}^3
\leq \|\phi\|_{M_A}\left(1+8CT_\ast A\|\phi\|_{M_A}^2\right)
\leq 2\|\phi\|_{M_A}
\end{align*} 
provided that $T_\ast\leq (8C)^{-1}A^{-1}\|\phi\|_{M_A}^{-2}$.
That is, $\Gamma$ maps the ball $B(0,2\|\phi\|_{M_A})$ into itself.
By making the constant $C$ in the above expression larger if needed,
we obtain similarly that $\Gamma$ is also a contraction of the ball $B(0,2\|\phi\|_{M_A})$.
This concludes the proof of the existence and uniqueness of the solution $u\in C([0,T_\ast]; M_A)$. 
The claim 
$u=\sum_{k=1}^\infty U_k[\phi]$ then follows immediately in the sense of the uniform convergence
of partial sums in $C([0,T_\ast];M_A)$.
\end{proof}

The following estimate of the $M_A$-norm
of $U_k[\phi]$ is useful in the proof of the norm inflation phenomenon. 
\begin{lemma}\label{LEM:UkMA}
There exists $C_2>0$ {\rm(}as in Lemma \ref{algebra}{\rm)}
such that for any $A\geq 1$, $k\geq 1$, and $\phi\in M_A$,
the following holds for all $t>0$:
\begin{equation}\label{UkMA}
\|U_k[\phi](t)\|_{M_A}\leq a_kt^{\frac{k-1}{2}}(C_2 A^{\frac 12}\|\phi\|_{M_A})^{k-1}\|\phi\|_{M_A},
\end{equation}
where $\{a_k\}_{k\in\mathbb N}$ is the sequence defined by 
\begin{equation*}
a_1=1, \quad\quad a_k=\frac{2}{k-1}\sum_{\substack{k_1,k_2, k_3\geq 1\\ k_1+k_2+k_3=k}}a_{k_1}a_{k_2}a_{k_3}, \quad k\geq 2.
\end{equation*}
\end{lemma}

The proof of Lemma \ref{LEM:UkMA} is
essentially the same as that of an analogous result in \cite{Kishimoto1}
(see also \cite{Iwabuchi_Ogawa}),
with the only difference that here we are using 
the unitarity of the operator $e^{-it|D|}$ in $M_A$, instead of $e^{it\pa_x^2}$. 
For the sake of completeness, we choose to reproduce this proof here.

\begin{proof}
The proof follows by induction. 
The case $k=1$ is trivial. 
Let us now assume that \eqref{UkMA}
holds for $1,2,\dots, k-1$, and let us prove it for $k$.
By the unitarity of $e^{-it|D|}$ in $M_A$, Lemma \ref{algebra},
and the induction hypothesis,
it follows that
\begin{align*}
\|U_k[\phi](t)\|_{M_A}
&\leq C_2^2A\sum_{\substack{k_1,k_2, k_3\geq 1\\ k_1+k_2+k_3=k}}
\int_0^t\|U_{k_1}[\phi]\|_{M_A}\|U_{k_2}[\phi]\|_{M_A}\|U_{k_3}[\phi]\|_{M_A}d\tau\\
&\leq C_2^2A \frac{2}{k-1}\sum_{\substack{k_1,k_2, k_3\geq 1\\ k_1+k_2+k_3=k}}a_{k_1}a_{k_2}a_{k_3}
t^{\frac{k-1}{2}}(C_2 A^{\frac 12}\|\phi\|_{M_A})^{k-3}\|\phi\|_{M_A}^3\\
&\leq a_kt^{\frac{k-1}{2}}(C_2 A^{\frac 12}\|\phi\|_{M_A})^{k-1}\|\phi\|_{M_A}.
\end{align*}
\end{proof}

In order to bound sequences $\{a_k\}_{k\in\mathbb N}$
with similar properties to the one in Lemma \ref{LEM:UkMA},
we use the following Lemma from \cite{Kishimoto1}. 
\begin{lemma}[\cite{Kishimoto1}]\label{LEM:ak}
Let $\{a_k\}_{k\in\mathbb N}$ be a sequence of nonnegative real numbers for which there exists $C>0$
such that 
\begin{equation*}
a_k\leq C\sum_{\substack{k_1,k_2, k_3\geq 1\\ k_1+k_2+k_3=k}}a_{k_1}a_{k_2}a_{k_3}
\end{equation*}
for all $k\geq 2$. Then, the following holds:
\begin{equation*}
a_k\leq C_4^{k-1}a_1^k
\end{equation*}
for all $k\geq 1$, where $C_4=\frac{\pi^2}{6}(9C)^{\frac 12}$.
\end{lemma}
The proof of Lemma \ref{LEM:ak}
is elementary (by induction) and details can be found in \cite{Kishimoto1}.
In particular, by Lemma \ref{LEM:ak}, it follows that the sequence $\{a_k\}_{k\in\mathbb N}$ from 
Lemma \ref{LEM:UkMA} satisfies $a_k\leq C_4^{k-1}$.

In the proof of Theorem \ref{THM:main} (iv)
we will work with an initial datum $\phi$ such that
\begin{equation}\label{phi}
\ft \phi (\xi):=R\left(\pmb{1}_{N+I_A}(\xi)+\pmb{1}_{2N+I_A}(\xi)\right) \text{ for all } \xi\in\R,
\end{equation}
where $N\gg1$, $1\ll A\ll N$, $R>0$ will be chosen later.
In other words, $\hat \phi$ is supported on two relatively small intervals centered at high frequencies $N, 2N \gg 1$.
It is useful to have 
the following estimate on the measure of the support of the Fourier transform of
$U_k[\phi]$.
\begin{lemma}
Let $A\geq 1$ and define $\phi$ as in \eqref{phi}. 
Then, there exists an absolute constant $C>0$
such that
\begin{equation}\label{supp}
|{\rm supp}\, \ft{U_k[\phi](t)}|\leq C^kA
\end{equation}
for any $k\geq 1$ and $t\geq 0$. 
In particular, the bound in \eqref{supp} is independent of $N$. 
\end{lemma}

\begin{proof}
The proof is essentially the same as that of an analogous result in \cite{Kishimoto1}.
Therefore, we only sketch it here and refer the readers to
\cite{Kishimoto1} for details. 
For $k$ even, \eqref{supp} is trivial since, as noticed earlier, $U_k[\phi]\equiv 0$.
For $k=1$, \eqref{supp} follows easily from the fact that $\ft\phi$ is supported on two intervals of size $A$
centered at $N$ and $2N$ respectively. 
For $k=3$, we notice that $\ft{U_3[\phi]}$ is supported on intervals centered at
$q=q_1-q_2+q_3$ with $q_1,q_2,q_3\in \{N,2N\}$ of size at most $3A$. Therefore,
\begin{equation}
|{\rm supp}\, \ft{U_3[\phi](t)}|\leq 2^3\cdot 3A. 
\end{equation}
Arguing by induction, it follows that $\ft{U_k[\phi](t)}$
is supported on at most $2^k$ intervals centered at integers, each of size at most $kA$ and thus,
$|{\rm supp}\, \ft{U_k[\phi](t)}|\leq 2^k\cdot kA$ for all $k\geq 3$.
\end{proof}

We are now ready to state and prove the norm inflation property
of \eqref{HW} in $H^s(\R)$, $s<0$, that we recall here for convenience.
\begin{proposition}[Norm inflation property for NHW in $H^s(\R)$, $s<0$]\label{PROP:inflation}
Let $s<0$. Then, given $0<\eps\ll 1$, there exist $\phi\in H^\infty (\R)$ with $\|\phi\|_{H^s}<\eps$
and $0<T<\eps$ such that the solution $u$ of \eqref{HW} with initial condition $u(0)=\phi$
satisfies $\|u(T)\|_{H^s}>\frac 1\eps$.
\end{proposition}

As already mentioned above, we choose the initial datum $\phi$ as in \eqref{phi}.
Note that $\|\phi\|_{M_A}=CRA^{\frac 12}$ and $\|\phi\|_{H^s}=C'RA^{\frac 12} N^s$.
The strategy of the proof of 
Proposition \ref{PROP:inflation}
is to expand the solution $u$ into the series of $U_k[\phi]$ as in \eqref{expansion},
and to show that the term $U_3[\phi]$ is much bigger
than all the other terms in the series. 
The conclusion then follows by choosing $R$, $T$, and $A$ 
conveniently in terms of $N$, such that $\|\phi\|_{H^s}\ll 1$, while $\|U_3[\phi]\|_{H^s}\gg 1$
for a fixed $s<0$.

The proof of 
Proposition \ref{PROP:inflation}
is based on the following two main lemmas.
The fist lemma gives an upper bound on the $H^s$-norm
of $U_k[\phi]$ for $k\in\mathbb N$ and $s\leq 0$.
\begin{lemma}
Let $s\leq 0$. Then there exists $C>0$ such that the following hold:
\begin{equation}\label{1}
\|U_1[\phi](t)\|_{H^s}\leq CRA^{\frac 12}N^s
\end{equation}
and
\begin{align}\label{2}
\|U_k[\phi](t)\|_{H^s}\leq t^{\frac{k-1}{2}}(CRA)^{k-1}R g(A),
\end{align}
for all $k\geq 2$ and all $t\geq 0$, where
\begin{align}\label{2bis}
g(A):=
\begin{cases}
A^{s+\frac 12}, & \text{if } -\frac 12<s\leq 0,\\
(\log A)^{\frac 12}, & \text{if } s=-\frac 12,\\
1, & \text{if } s<-\frac 12.
\end{cases}
\end{align}

\end{lemma}

\begin{proof}
The first estimate \eqref{1} is trivial, so we concentrate on \eqref{2}. 
In what follows, $C>0$ 
denotes a generic constant (possibly increasing from line to line).
By H\"older's and Young's inequalities, we have that
\begin{align}\label{a}
&\|U_k[\phi](t)\|_{H^s}\leq \|\jb{\xi}^s\pmb{1}_{{\rm supp}\, \ft{U_k[\phi](t)}}\|_{L^2}\|\ft{U_k[\phi](t)}\|_{L^\infty}\\
&\leq C\|\jb{\xi}^s\pmb{1}_{{\rm supp}\, \ft{U_k[\phi](t)}}\|_{L^2} 
\sum_{\substack{k_1,k_2, k_3\geq 1\\ k_1+k_2+k_3=k}}
\int_0^t \|\ft{U_{k_1}[\phi]}\|_{L^2}\|\ft{U_{k_2}[\phi]}\|_{L^2}\|\ft{U_{k_3}[\phi]}\|_{L^2}|{\rm supp}\, \ft{U_{k_3}[\phi](\tau)}|^{\frac 12}d\tau.\notag
\end{align}
By \eqref{supp}, we first notice that
\begin{align}\label{b}
\|\jb{\xi}^s\pmb{1}_{{\rm supp}\, \ft{U_k[\phi](t)}}\|_{L^2}
\lesssim
\begin{cases}
(C^kA)^{s+\frac 12}, & \text{if } \quad -\frac 12<s\leq 0,\\
(\log (C^kA))^{\frac 12}, & \text{if } \quad s=-\frac 12, \\
1, & \text{if } s<-\frac 12.
\end{cases}
\end{align}
Note that $(\log (C^kA))^{\frac 12}=(k\log C+\log A)^{\frac 12}\leq (C')^{k-1}(\log A)^{\frac 12}$.
Secondly, since $k_1,k_2\geq 1$ and $k_3\leq k-2$, we have by \eqref{supp} that
\begin{equation*}
|{\rm supp}\, \ft{U_{k_3}[\phi](\tau)}|\leq C^{k-2}A.
\end{equation*}
Thirdly, by Lemma \ref{algebra}, \eqref{UkMA}, and $\|\phi\|_{M_A}\leq CRA^{\frac 12}$, it follows that
\begin{align*}
\|U_{k_j}[\phi](\tau)\|_{L^2}&\leq C\|U_{k_j}[\phi](\tau)\|_{M_A}
\leq Ca_{k_j}t^{\frac{k_j-1}{2}}(C_2 A^{\frac 12}\|\phi\|_{M_A})^{k_j-1}\|\phi\|_{M_A}\\
&\leq Ca_{k_j}t^{\frac{k_j-1}{2}}(CRA)^{k_j-1}RA^{\frac 12}.
\end{align*}
Combining the last two estimates and using $a_k\leq C_4^{k-1}$ from Lemma \ref{LEM:ak}, we obtain that
\begin{align}\label{c}
\sum_{\substack{k_1,k_2, k_3\geq 1\\ k_1+k_2+k_3=k}}
&\int_0^t \|\ft{U_{k_1}[\phi]}\|_{L^2}\|\ft{U_{k_2}[\phi]}\|_{L^2}\|\ft{U_{k_3}[\phi]}\|_{L^2}|{\rm supp}\, \ft{U_{k_3}[\phi](\tau)}|^{\frac 12}d\tau\notag \\
&\leq (C^{k-2}A)^{\frac 12} \frac{2}{k-1}\sum_{\substack{k_1,k_2, k_3\geq 1\\ k_1+k_2+k_3=k}}a_{k_1}a_{k_2}a_{k_3} t^{\frac{k-1}{2}}C^3(CRA)^{k-3}(RA^{\frac 12})^3\notag \\
&\leq a_k t^{\frac{k-1}{2}} (CRA)^{k-1}R \leq t^{\frac{k-1}{2}} (CRA)^{k-1}R.
\end{align}
The conclusion then follows from \eqref{a}, \eqref{b}, and \eqref{c}.
\end{proof}
Next, we prove a lower bound for the $H^s$-norm of $U_3[\phi]$ for $s\leq 0$.
\begin{lemma}\label{LEM:U3}
Let $s\leq 0$. Then, there exists $C>0$ such that for all $t>0$, the following holds
\begin{align}\label{3}
\|U_3[\phi](t)\|_{H^s}\geq CtR^3A^2 g(A),
\end{align}
where $g(A)$ was defined in \eqref{2bis}.
\end{lemma}

\begin{proof}
We write
\begin{align*}
\ft{U_3[\phi](t)}(\xi)&=-i\mu e^{-it|\xi|}\int_0^te^{i\tau |\xi|}\mathcal{F}\Big(|e^{-i\tau|D|}\phi|^2e^{-i\tau|D|}\phi\Big)(\xi)d\tau \\
&=-i\mu e^{-it|\xi|}\iint\left(\int_0^te^{i\tau(|\xi|-|\xi_1|+|\xi_2|-|\xi_3|)}d\tau\right)\ft{\phi}(\xi_1)\cj{\ft{\phi}(\xi_2)}\ft{\phi}(\xi_3)\pmb{1}_{\xi=\xi_1-\xi_2+\xi_3}d\xi_1d\xi_3.
\end{align*}
From the definition \eqref{phi} of $\phi$, we notice that $\ft\phi$ is supported only on positive frequencies. 
Therefore, the expression under the above integral is supported on $\xi_1,\xi_2,\xi_3\geq 0$.
Next, we restrict our attention to $\xi\in [0,\frac A8)$. In particular, for such $\xi$, we have
\[|\xi|-|\xi_1|+|\xi_2|-|\xi_3|=\xi-\xi_1+\xi_2-\xi_3=0.\]
Noticing also that $\xi\in [0,\frac A8)\subset I_{\frac A4}$ and $\xi_1,\xi_3\in N+I_{\frac A4}$ yield $\xi_2=\xi_1+\xi_3-\xi\in 2N+I_A$ (and thus $\pmb{1}_{2N+I_A}(\xi_2)\equiv 1$), we then obtain that

\begin{align*}
\left|\ft{U_3[\phi](t)}(\xi)\pmb{1}_{[0,\frac A8)}(\xi)\right| &
=t\Big|\iint\ft{\phi}(\xi_1)\cj{\ft{\phi}(\xi_2)}\ft{\phi}(\xi_3)\pmb{1}_{\xi=\xi_1-\xi_2+\xi_3}
\pmb{1}_{[0,\frac A8)}(\xi)d\xi_1d\xi_3\Big|\\
&=tR^3\iint\pmb{1}_{N+I_A}(\xi_1)\pmb{1}_{2N+I_A}(\xi_2)\pmb{1}_{N+I_A}(\xi_3)\pmb{1}_{\xi=\xi_1-\xi_2+\xi_3}
\pmb{1}_{[0,\frac A8)}(\xi)d\xi_1d\xi_3\\
&\geq tR^3\int_\R \pmb{1}_{N+I_{\frac A4}}(\xi_1)d\xi_1 \int_\R \pmb{1}_{N+I_{\frac A4}}(\xi_3)d\xi_3
\geq CtR^3A^2.
\end{align*}
In conclusion, it follows that
\begin{align*}
\|U_3[\phi](t)\|_{H^s}&\geq \|\jb{\xi}^s\ft{U_3[\phi](t)}(\xi)\pmb{1}_{[0,\frac A8)}(\xi)\|_{L^2}
\geq CtR^3A^2\|\jb{\xi}^s\pmb{1}_{[0,\frac A8)}\|_{L^2}\\
&\geq CtR^3A^2 g(A).
\end{align*}

\end{proof}

We are now in the position of proving Proposition \ref{PROP:inflation}.
\begin{proof}[Proof of Proposition \ref{PROP:inflation}]
We choose $\phi$ as in \eqref{phi} with the values of $R,T,A$ to be specified later.
The condition $\|\phi\|_{H^s}<\eps$ with $0<\eps\ll 1$ is satisfied if we impose
\begin{equation}\label{4}
RA^{\frac 12}N^s\ll 1. 
\end{equation}
By Lemma \ref{LWP},
there exists a unique solution $u\in C([0,T_\ast];M_A)$
of \eqref{HW} admitting the expansion \eqref{expansion},
where $T_\ast\sim  CA^{-1}\|\phi\|_{M_A}^{-2}=CA^{-2}R^{-2}$.  
We then require that the time $T$ in Proposition \ref{PROP:inflation}
satisfies
\begin{equation}\label{5}
T\leq T_\ast\sim A^{-2}R^{-2}.
\end{equation}
Under the restrictions \eqref{4} and \eqref{5}, we now impose that 
\begin{align}
\|U_3[\phi](T)\|_{H^s}&\gg \|U_1[\phi](T)\|_{H^s}\label{8}\\
\|U_3[\phi](T)\|_{H^s}&\gg \sum_{\ell=2}^\infty \|U_{2\ell+1}[\phi] (T)\|_{H^s},\label{9}
\end{align}
in order to have $\|u(T)\|_{H^s}\gtrsim \|U_3[\phi](T)\|_{H^s}$. 
The conclusion of the proposition follows if we further impose
that 
\begin{align}\label{10}
\|U_3[\phi](T)\|_{H^s}\gg 1.
\end{align}
Owing to \eqref{1} and \eqref{3}, 
the condition \eqref{8} amounts to
\begin{align}\label{6}
RA^{\frac 12}N^s \ll TR^3A^2 g(A)
\end{align}
and \eqref{10} amounts to
\begin{align}\label{7}
1\ll TR^3A^2 g(A),
\end{align}
where $g(A)$ was defined in \eqref{2bis}.
On the other hand, 
by \eqref{2},
note that $\sum_{\ell=2}^\infty \|U_{2\ell+1}[\phi](T)\|_{H^s}$
behaves like the geometric series
\begin{align*}
\sum_{\ell=2}^\infty (CTR^2A^2)^\ell R g(A).
\end{align*}
Thus, the series converges and the first term dominates all others provided that $TR^2A^2\ll 1$. 
Combining this with \eqref{3}, it follows that \eqref{9} is satisfied provided that
\[TR^2A^2\ll 1.\]
To summarize, the conclusion 
of Proposition \ref{PROP:inflation} follows if one can choose $1\ll A\ll N$, $T\ll 1$, and $R$
such that
\begin{align}\label{11}
TR^2A^2+RA^{\frac 12}N^s\ll 1\ll
TR^3A^2
\begin{cases}
A^{s+\frac 12}, & \text{if }  -\frac 12<s\leq 0,\\
(\log A)^{\frac 12}, & \text{if }  s=-\frac 12, \\
1, & \text{if } s<-\frac 12.
\end{cases}
\end{align}
For $A\gg 1$ and $s\leq 0$, a stronger condition than \eqref{11} is: 
\begin{align}\label{12}
TR^2A^2+RA^{\frac 12}N^s\ll 1\ll
TR^3A^2.
\end{align}
Therefore, Proposition \ref{PROP:inflation} also follows if one can choose $A,T,R$
satisfying the stronger condition:
\begin{equation}\label{12bis}
T\ll1\,,\quad 1\ll A\ll N\,,\quad
RA^\frac12N^s\ll1\,,\quad
\frac T{R^2}(RA^\frac12)^4\ll1\ll \frac TR(RA^\frac12)^4\,.
\end{equation}
Set now
\begin{equation}RA^\frac12=N^\theta\,,\quad T=N^a\,,\quad R=N^b\,,\quad {\rm so~that}\quad A=N^{2\theta-2b}.
\label{param}
\end{equation}
Then, the conditions \eqref{12bis} are satisfied 
exactly when $\theta<-s$ and 
\begin{equation}
\left.\begin{array}{rcl}
a&<&0\\
\max\{0\,,\,\theta-\frac12\}<b&<&\theta\\
a-2b+4\theta & <& 0\\
a-b+4\theta & > & 0
\end{array}\right\}
\label{geomHW}
\end{equation}
In particular, notice that this imposes $0<\theta<-s$, which is possible provided that $s<0$.  
Choosing, for example, $0<\theta<\min(-s,\frac12)$
it follows that the solution of \eqref{geomHW} is nonempty.
Namely, it is the interior of the triangle $ABC$ (see Figure \ref{Fig2} below)
whose vertices have $(a,b)$-coordinates given by
\begin{equation}
A\colon(-4\theta,0)\,,\quad B\colon(-3\theta,\theta)\,,\quad C\colon(-2\theta,\theta).
\label{tri.vert}
\end{equation}

\begin{figure}[h]
\includegraphics{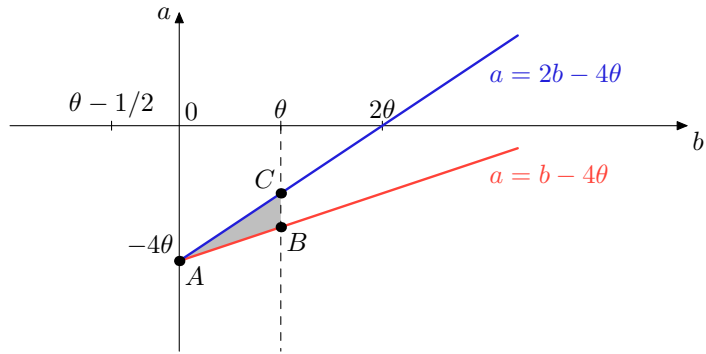}
\caption{Triangle $ABC$.}
\label{Fig2}
\end{figure}

\noi
This shows that for $s<0$ it is indeed possible to choose
$T,R$, and $A$ as in \eqref{param}, satisfying \eqref{12bis}.
Therefore, \eqref{HW} has the norm inflation property in $H^s(\R)$, $s<0$,
which concludes the proof of Proposition \ref{PROP:inflation}. 

\end{proof}

\begin{remark}\label{REM:inflation SZ}
{\rm
(i). Due to the conservation of the $L^2$-norm by the flow of \eqref{HW},
Proposition \ref{PROP:inflation} does not hold for $s=0$.

\smallskip
\noi
(ii). The cubic Szeg\H{o} equation also has the norm inflation property in $H^s_+(\R)$, $s<0$. 
The proof is essentially the same as that of Proposition \ref{PROP:inflation} , with the exception that the operator $e^{-it|D|}$ is replaced by the identity operator. 
}
\end{remark}

\section{Norm inflation for fractional NLS in $H^s(\R)$, $s<0$}\label{sec:FNLS}

We end this paper 
with the proof of Theorem \ref{fracNLS}.
Namely, we show that the norm inflation phenomenon in $H^s(\R)$, $s<0$, also occurs 
for other cubic fractional nonlinear Schr\"odinger equations.

\begin{proof}[Proof of Theorem \ref{fracNLS}]
The proof follows the same lines as that of Proposition \ref{PROP:inflation},
with the operator $e^{-it|D|}$
replaced by $e^{-it|D|^\beta}$. 
In particular, 
we consider the same initial condition $u(0)=\phi$
as in \eqref{phi}, the local existence time $T_\ast\sim A^{-2}R^{-2}$ is the same as for \eqref{HW},
and the solution of \eqref{FNLS0} is again given by an expansion $u=\sum U_k[\phi]$.
Evidently, the condition $\|\phi\|_{H^s}\ll 1$ remains $RA^{\frac 12}N^s\ll 1$.
The only essential difference appears in the proof of Lemma \ref{LEM:U3}. 
More precisely, the time integral in the proof of Lemma \ref{LEM:U3} becomes
\[\int_0^t e^{i\tau (|\xi|^\beta-|\xi_1|^\beta+|\xi_2|^\beta-|\xi_3|^\beta)}d\tau\]
with $\xi-\xi_1+\xi_2-\xi_3=0$ and $\xi_1,\xi_2,\xi_3>0$.
For $\beta\neq 1$, one can no longer choose $\xi$ so that the above phase vanishes. 
Instead, we notice that for $\xi\in I_{\frac A4}$, $\xi_1,\xi_3 \in N+I_{A}$, and $\xi_2\in 2N+I_A$, 
one has $|\xi|^\beta-|\xi_1|^\beta+|\xi_2|^\beta-|\xi_3|^\beta=O(N^\beta)$. Therefore, if we choose $|t|\ll N^{-\beta}$, we obtain 
\[\Big|\int_0^t e^{i\tau (|\xi|^\beta-|\xi_1|^\beta+|\xi_2|^\beta-|\xi_3|^\beta)}d\tau\Big|\gtrsim t,\]
which suffices for the purposes of Lemma \ref{LEM:U3}. 
The extra condition $T\ll N^{-\beta}$, however, needs to be added to the 
previous requirements $1\ll A\ll N$ and \eqref{11} from the proof of Proposition  \ref{PROP:inflation}.

\noindent
\underline{Case I}: $s\leq -\frac 12$.

As in the proof of Proposition \ref{PROP:inflation},
we verify the stronger conditions
\eqref{12bis} instead of \eqref{11}, to which we add the requirement $T\ll N^{-\beta}$.
Setting
\begin{equation*}
RA^\frac12=N^\theta\,,\quad T=N^a\,,\quad R=N^b\,,\quad {\rm so~that}\quad A=N^{2\theta-2b},
\end{equation*}
as in the proof of Proposition \ref{PROP:inflation},
these are satisfied 
exactly when $\theta<-s$ and 
\begin{equation}
\left.\begin{array}{rcl}
a&<&-\beta\\
\max\{0\,,\,\theta-\frac12\}<b&<&\theta\\
a-2b+4\theta & <& 0\\
a-b+4\theta & > & 0
\end{array}\right\}
\label{geomFNLS1}
\end{equation}

\noindent
\underline{Subcase I.a}: $s\leq -\frac 12$ and $\theta\geq \frac 12$.

In this subcase, the solution of \eqref{geomFNLS1}
is the intersection of 
the half plane $a<-\beta$
with
the interior of the quadrilateral $BCDE$
\begin{figure}[ht]
\includegraphics{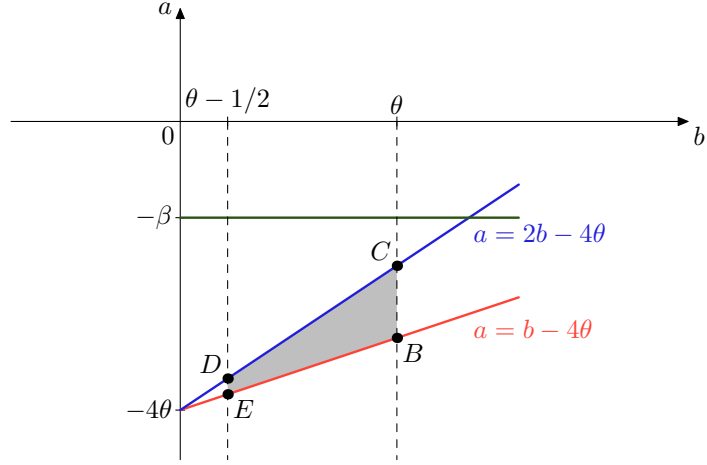}
\caption{Quadrilateral $BCDE$.}
\label{Fig3}
\end{figure}
in Figure \ref{Fig3}, whose vertices have $(a,b)$-coordinates given by
\begin{equation*}
B\colon (-3\theta,\theta)\,,\quad C\colon (-2\theta,\theta)\,,\quad
D\colon (-2\theta-1,\theta-\frac12)\,,\quad E\colon (-3\theta-\frac12,\theta-\frac12).\,
\end{equation*}
This intersection is nonempty
if the vertex $E$ of the quadrilateral $BCDE$ lies in the half plane $a<-\beta$,
that is if $-\beta>-3\theta -\frac 12$.
This imposes the following constraints on $\theta$: 
$\frac{2\beta-1}{6}<\theta<-s$ and $\theta\geq \frac 12$. 
Such $\theta$ exists provided that either $0<\beta\leq 2$ and $s<-
\frac 12$, or $\beta>2$ and $s<\frac{1-2\beta}{6}$. 

In conclusion, for
$0<\beta\leq 2$ and $s<-
\frac 12$,
or
$\beta>2$ and $s<\frac{1-2\beta}{6}$, 
conditions \eqref{12bis} and $T\ll N^{-\b}$ are indeed satisfied.

\medskip
\noindent
\underline{Subcase I.b}: $s\leq -\frac 12$ and $0<\theta< \frac 12$.

The solution of \eqref{geomFNLS1}, in this subcase,
is the intersection of the half plane $a<-\beta$ with
the interior of the triangle $ABC$ (see Figure \ref{Fig2} above)
whose vertices have $(a,b)$-coordinates given by
\begin{equation*}
A\colon(-4\theta,0)\,,\quad B\colon(-3\theta,\theta)\,,\quad C\colon(-2\theta,\theta).
\end{equation*} 
This intersection is nonempty
if the vertex $A$ of the triangle lies in the half plane $a<-\beta$,
that is $-\beta>-4\theta$.
This dictates the choice $\frac{\beta}{4}<\theta< \frac 12$,
which is possible only if $\beta<2$.  

Therefore, we obtain that if $0<\beta<2$ and $s\leq -\frac 12$
conditions \eqref{12bis} and $T\ll N^{-\b}$ are satisfied. 

\medskip
\noindent
\underline{Case II}: $s=-\frac 12$ and $\beta=2$.

In this case, 
instead of 
verifying the stronger conditions \eqref{12bis},  
we verify conditions \eqref{11}.
We choose, as in \cite{Kishimoto1}:
\begin{align*}
T=\frac{1}{N^2(\log N)^{\frac 16}}, \qquad R=1, \qquad A=\frac{N}{(\log N)^{\frac{1}{12}}}.
\end{align*}
Then, $RA^{\frac 12}N^s=\frac{1}{(\log N)^{\frac{1}{24}}}\ll 1$, $TR^2A^2=\frac{1}{(\log N)^{\frac 13}}\ll 1$, and
$TR^3A^2\cdot (\log A)^{\frac 12}\sim (\log N)^{\frac 16}\gg 1$. 
This shows that $T\ll N^{-2}$, $1\ll A\ll N$, and \eqref{11} are indeed satisfied.

\medskip
\noindent
\underline{Case III}: $-\frac 12 <s<0$.

From \eqref{11}, 
it follows that in this case it is enough to verify the following conditions:
\begin{align}\label{12bisbis}
T\ll N^{-\beta}\,,\quad 1\ll A\ll N\,,\quad
TR^2A^2+RA^{\frac 12}N^s\ll 1\ll
TR^3A^{\frac 52+s}.
\end{align}
With the choice of $T$, $R$, and $A$ from \eqref{param},
it follows that
these conditions are satisfied 
exactly when $\theta<-s$ and 
\begin{equation}
\left.\begin{array}{rcl}
a&<&-\beta\\
\theta-\frac12<b&<&\theta\\
a-2b+4\theta & <& 0\\
a-(2+2s)b+(5+2s)\theta & > & 0
\end{array}\right\}
\label{geomFNLS2}
\end{equation}
Notice that $\theta<-s$ and $s>-\frac 12$ yield $\theta<\frac 12$. 

Secondly, we notice that the above conditions cannot be simultaneously satisfied 
if $\theta\leq 0$. Indeed, assume that $\theta\leq 0$.
The last two conditions in \eqref{geomFNLS2}
require that:
\[(2+2s)b-(5+2s)\theta<a<2b-4\theta,\]
and thus $b>\frac{1+2s}{2s}\theta\geq 0$. 
On the other hand, the second condition in 
\eqref{geomFNLS2} would impose that $b\in (\theta-\frac 12, \theta)\subset (-\infty, 0)$.
This is a contradiction, and therefore \eqref{geomFNLS2} does not have a solution if $\theta\leq 0$.

From now on we assume that $0<\theta<\frac 12$.
Then, the solution of \eqref{geomFNLS2}
is the intersection of the half plane $a<-\beta$
with the interior of the triangle $A'BC$ 
in Figure \ref{Fig4} below,
\begin{figure}[h]
\includegraphics{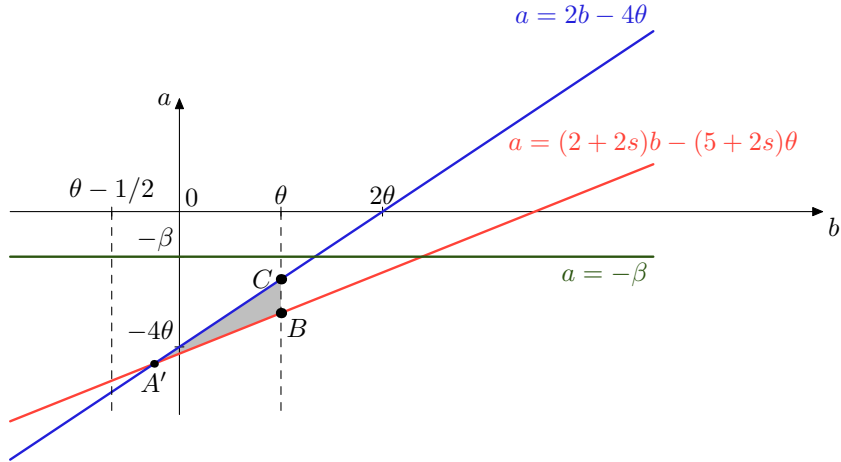}
\caption{Triangle $A'BC$.}
\label{Fig4}
\end{figure}

\noi
whose vertices have $(a,b)$-coordinates given by
\begin{equation*}
A'\colon \left(\frac{1-2s}{s}\theta, \frac{1+2s}{2s}\theta\right)\,,\quad B\colon(-3\theta,\theta)\,,\quad C\colon(-2\theta,\theta).
\end{equation*} 
This intersection is nonempty
if the vertex $A'$ of the triangle lies in the half plane $a<-\beta$,
that is if $-\beta>\frac{1-2s}{s}\theta$.
Combining this with $\theta<-s$ and $-\frac 12<s<0$ yields $-\frac 12< s<\min(\frac{1-\beta}{2},0)$ and $\beta<2$. 

Therefore, 
if $0<\beta<2$ and $-\frac 12< s<\min(\frac{1-\beta}{2},0)$, 
the conditions \eqref{12bisbis} are satisfied. 

\medskip
Collecting the information we obtained in Cases I, II, and III above,
we conclude that the norm inflation phenomenon occurs for \eqref{FNLS0} 
in the following cases:
$0<\b<1$ and $s<0$; $1\leq \beta<2$ and $s<s_{\rm crit}$; 
$\b=2$ and $s\leq s_{\rm crit}$; $\beta>2$ and $s<\frac{1-2\b}{6}$.

\end{proof}

\begin{remark}
{\rm
The method developed in this section
cannot be used to decide on the norm inflation property of \eqref{FNLS0} 
in $H^s(\R)$ in the following two cases:
$1<\b<2$ and
 $s=s_{\rm crit}$; $\b>2$ and $s=\frac{1-2\b}{6}$.
 }
\end{remark}

\begin{proof}
We follow an argument from \cite[Remark B.5]{Oh}.
In view of the requirements $A\ll N$ and $T\ll N^{-\beta}$ of our method, 
we write  
$A=EN$ with $E\ll 1$ and $T=FN^{-\beta}$ with $F\ll 1$.

For $1<\b<2$, notice that $-\frac 12<s_{\rm crit}=\frac{1-\b}{2}<0$. 
Set $G:=RA^{\frac 12}N^{s_{\rm crit}}\ll 1$.
Then,
\begin{align*}
TR^3A^{\frac 52+s_{\rm crit}}=FN^{-\beta}(RA^{\frac 12}N^{s_{\rm crit}})^3N^{-3s_{\rm crit}}A^{1+s_{\rm crit}}
=E^{1+s_{\rm crit}}FG^3\ll 1.
\end{align*}
In particular, there is no choice of $T$, $R$, and $A$ such that 
$RA^{\frac 12}N^{s_{\rm crit}}\ll 1$ and
$TR^3A^{\frac 52+s}\gg 1$,
as required by conditions \eqref{11} for the norm inflation property.

For $\b>2$, notice that $\frac{1-2\b}{6}<-\frac 12$.
Set $H:=TR^3A^2\gg 1$.
This yields $R=E^{-\frac 23}F^{-\frac 13}H^{\frac 13}N^{\frac{\b-2}{3}}$, and thus
\[RA^{\frac 12}N^{\frac{1-2\b}{6}}=E^{-\frac 16}F^{-\frac 13}H^{\frac 13}\gg 1.\]
In particular, there is no choice of $T$, $R$, and $A$ such that 
$TR^3A^2\gg 1$ and
$RA^{\frac 12}N^{\frac{1-2\b}{6}}\ll 1$, as required by conditions \eqref{11} 
for the norm inflation property.
\end{proof}

\begin{ackno}
{\rm 
A.C. was partially supported 
by a Whittaker Research Fellowship at the University of Edinburgh and 
by the European Research Council (grant agreement no. 616797).
O.P. was partially supported by the National Science Foundation under grant no. DMS-1440140 while she was in residence at the Mathematical Sciences Research Institute in Berkeley, California, during the Fall 2015 semester.
The authors would like to thank Tadahiro Oh, Nobu Kishimoto, Nikolay Tzvetkov, and Vladimir Georgiev for their generous help.
}
\end{ackno}



\end{document}